\documentclass[12pt,a4paper,reqno,oneside]{amsart}

\usepackage{amsmath,amssymb,euscript,amsthm,amsfonts,hyperref}

\usepackage[usenames]{color}
\theoremstyle{plain}
\newtheorem{thm}{Theorem}[section]
\newtheorem{lem}{Lemma}[section]
\newtheorem{assum}{Assumption}[section]
\newtheorem{corl}{Corollary}[section]

\theoremstyle{definition}

\theoremstyle{remark}
\newtheorem{remk}{Remark}[section]

\newtheorem{expl}{Example}[section]

\newcommand{\mbR}{{\mathbb R}}
\newcommand{\mbN}{{\mathbb N}}

\newcommand{\ba}{\begin{aligned}}
\newcommand{\ea}{\end{aligned}}

\newcommand{\wt}{\widetilde}

\newcommand{\vf}{\varphi}

\newcommand{\ve}{\varepsilon}
\renewcommand{\lg}{\langle}
\newcommand{\rg}{\rangle}

\newcommand{\const}{\mathop{\rm const}}

\newcommand{\Pb}{\mathrm{P}} 
\newcommand{\E}{\mathrm{E}} 
\newcommand{\be}{\begin{equation}}
\newcommand{\bel}{\begin{equation}\label}
\newcommand{\ee}{\end{equation}}

\begin{document}

\title[Small Noise Perturbations in
Multidimensional Case]{Small Noise Perturbations in
Multidimensional Case}

\date{\today}

\author{Andrey Pilipenko}
\address{Institute of Mathematics,  National Academy of Sciences of
Ukraine, Tereshchenkivska str. 3, 01601, Kiev, Ukraine; National Technical University of Ukraine 
``Igor Sikorsky Kyiv Polytechnic Institute'',
ave.\ Pobedy 37, Kiev 03056, Ukraine.}

\author{Frank Norbert Proske}
\address{Department of Mathematics, University of Oslo, PO Box 1053 Blindern, N-316 Oslo, Norway}


 \begin{abstract}
In this paper we study zero-noise limits of $\alpha -$stable noise perturbed
ODE's which are driven by an irregular vector field $A$ with asymptotics $%
A(x)\sim \overline{a}(\frac{x}{\left\vert x\right\vert })\left\vert
x\right\vert ^{\beta -1}x$ at zero, where $\overline{a}>0$ is a continuous
function and $\beta<1$. The results established in this article can
be considered a generalization of those in the seminal works of Bafico \cite%
{Ba} and Bafico, Baldi \cite{BB} to the multi-dimensional case.\ Our
approach for proving these results is inspired by techniques in \cite%
{PP_self_similar} and based on the analysis of an SDE for $t\longrightarrow
\infty $, which is obtained through a transformation of the perturbed ODE.  
\end{abstract}

\maketitle

\section{Introduction}
In order to illustrate small noise analysis for singular ordinary
differential equations (ODE's) from an application's point of view, let us
consider the stochastic gradient method in non-convex optimization:

In deep learning the problem of training neural networks with a
training set can be translated into a problem of the minimization of a
non-convex loss function $F:\mathbb{R}^{d}\longrightarrow \mathbb{R}$. A
popular approach for solving this minimization problem is the
stochastic gradient algorithm (SGD), which is numerically more tractable
than the classical gradient descent method in the case of high dimensional
data; see e.g. \cite{Bu}, \cite{Bot} and the references therein. 

To give more details about this approach, let $f:\mathbb{R}^{d}\times 
\mathbb{R}^{l}\longrightarrow \mathbb{R}$ and $\xi $ be a random variable in $%
\mathbb{R}^{l}$ such that $f(\cdot ,y)$ is twice differentiable for all $%
y\in \mathbb{R}^{l}$ and $F(x)=\E[f(x,\xi )],x\in \mathbb{R}^{d}$. Then a
sequence $x_{k},k\geq 0$, which converges to a local minimizer of $F$, is
constructed recursively via%
\begin{equation*}
x_{k}^{\varepsilon }=x_{k-1}^{\varepsilon }-\varepsilon \nabla f(x_k^\ve,\xi
_{k}),x_{0}^{\varepsilon }=\xi _{0},
\end{equation*}%
where $\varepsilon >0$ is the \emph{learning rate}, $\xi _{k},k\geq 0$ an $%
i.i.d-$sequence of random variables and $\nabla =\nabla _{x}$ the gradient.

Here, a central issue is the study of the problem of how fast the iterations $x_{k}^{\varepsilon },k\geq 0$ can escape
from unstable stationary points, that is from points $x$ with $\nabla F(x)=0$
such that the least eigenvalue of the Hessian $\nabla ^{2}F(x)$ is strictly
negative. More recently (see \cite{HuLi, Feng, LiMalladi}), it could be
shown that the speed of escape from non-stable stationary points can be
analyzed by means of solutions $X_{t},0\leq t\leq T$ to stochastic
differential equations (SDE's) which approximate $x_{k}^{\varepsilon },k\geq
0$ in a certain weak sense. More precisely, it was proven in \cite{HuLi, LiMalladi} for
sufficiently smooth $f$ 
and positive semi-definite
functions $\sigma :\mathbb{R}^{d}\longrightarrow \mathbb{R}^{d\times d}$
that the solution $X_{\cdot }$ to the SDE%
\begin{equation}
dX_{t}^{\varepsilon }=-\nabla F(X_{t}^{\varepsilon })dt+\sqrt{\varepsilon }%
\sigma (X^{\varepsilon }(t))dW_{t},0\leq t\leq T,X_{0}=x_{0},  \label{SDE}
\end{equation}%
for a Wiener process $W_{\cdot }$, approximates $x_{k}^{\varepsilon },k\geq 0$ in the following sense: For all 
sufficiently smooth $%
\varphi$
there exist constants $C,\varepsilon _{0}>0$
(depending on $T$ and $\varphi $) such that%
\begin{equation*}
\left\vert E[\varphi (x_{k}^{\varepsilon })]-E[\varphi (X_{k\varepsilon
}^{\varepsilon })]\right\vert \leq C\varepsilon ,k=0,...,\left[ \frac{T}{%
\varepsilon }\right] ,\varepsilon <\varepsilon _{0}\text{.}
\end{equation*}

\bigskip On the other hand, one may encounter in applications the situation
that the coefficients $A:=-\nabla F$ and $\sigma $ in (\ref{SDE}) do not
meet the above smoothness requirements; see e.g. \cite{Ng}. For example it may happen that the vector
field $A$ is irregular, that is non-Lipschitz or discontinuous, while $%
\sigma $ is constant. From a practical and theoretical perspective, it would
be therefore important to extend those results to the case of
irregular vector fields $A$. However, in persuing such an objective one has
to cope with the possibility of the occurrence of an interesting effect in
connection with the small noise perturbed ODE (\ref{SDE}), namely the \emph{%
stochasticity} of a solution $X_{\cdot }$ to (\ref{SDE}) for $\varepsilon =0$
as the limit of $X_{\cdot }^{\varepsilon }$ for $\varepsilon \searrow 0$ in
distribution. To explain this phenomenon, consider the ODE 
\begin{equation}
X_{t}=x+\int_{0}^{t}A(X_{s})ds,x\in \mathbb{R}^{d},t\geq 0  \label{CP}
\end{equation}%
for $x\in \mathbb{R}^{d}$, where $A:\mathbb{R}^{d}\longrightarrow \mathbb{R}%
^{d}$ is a Borel measurable vector field.

If we assume that $A$ is Lipschitzian, it is well-known
that one can construct a global unique solution $X_{\cdot }\in C([0,\infty );%
\mathbb{R}^{d})$ to (\ref{CP}) by using Picard iteration.

On the other hand, if $A$ is not Lipschitzian, uniqueness or even existence
of solutions of the ODE (\ref{CP}) are not any longer guaranteed. An example
is the ODE (\ref{CP}) driven by the discontinuous vector field $A$ given by%
\begin{equation}
A(x)=sgn(x)  \label{sgn2}
\end{equation}%
for $X_{0}=0$, which possesses infinitely many solutions, where $%
X_{t}=+t,-t,t\geq 0$ are extremal solutions.

Other examples of non-well-posedness of the ODE (\ref{CP}) in the above
sense can be also observed in the case of continuous non-Lipschitz vector fields $A$
satisfying the growth condition $\left\langle A(x),x\right\rangle \leq
K(\left\vert x\right\vert ^{2}+1),x\in \mathbb{R}^{d}$ for a constant $K$.
Then, using e.g. Peano's theorem and the theorem of Arzel\`{a}-Ascoli one
finds that the set $C(x)$ of solutions $X_{\cdot }\in C([0,\infty );\mathbb{R%
}^{d})$ to (\ref{CP}) is non-empty, compact (in $C([0,\infty );\mathbb{R}%
^{d})$) and connected; see \cite{St}.

The case, when $C(x)$ is a singleton, that is the case of uniqueness of
solutions to (\ref{CP}) was e.g. examined \cite{DL}, \cite{A}, \cite{AG} by
means of the concept of renormalized solutions in connection with the
associated continuity equation.

\bigskip However, when $C(x)$ is not a singleton, which corresponds to the
situation of non-uniqueness, one may be faced with the problem of the
selection of the "right" or "most appropriate" solution to (\ref{CP}).

An important approach for studying such a selection problem based on so-called
Markov selections was developed by Krylov \cite{Kr}.

An alternative method to the latter one, which we aim at applying in this
paper, is based on the zero-noise selection principle, that is on the
selection of a solution $X_{\cdot }$ to (\ref{CP}) obtained as a limit in
law of solutions $X_{\cdot }^{\varepsilon }$ to a small noise perturbed ODE (%
\ref{CP}) of the form
\begin{equation}
X_{t}^{\varepsilon }=x+\int_{0}^{t}A(X_{s})ds+\varepsilon W_{t}
\label{SmallNoise}
\end{equation}%
for $\varepsilon \searrow 0$. We remark that in this case well-posedness of
the equation (\ref{CP}) is restored by adding a regularizing noise 
$\varepsilon W_{t}$, $\varepsilon >0$ to (\ref{CP}), since the SDE (\ref%
{SmallNoise}) has a unique strong solution or even a \emph{path-by-path}
unique solution in the space of continuous functions for each $\varepsilon >0
$, when e.g. $A$ is bounded and measurable (see \cite{Veretennikov} and \cite{Davie}).
The purpose of this method, which was originally proposed by A. N.\
Kolmogorov, is to select solutions to (\ref{CP}), which exhibit stability
under random perturbations.

When $Card(C(x))>1$, that is the case which is also referred to as Peano
phenomenon in the literature, Bafico \cite{Ba} and Bafico, Baldi \cite{BB}
were the first authors, who studied the zero-noise problem in the case of
one-dimensional time-homogeneous vector fields $A$. They could show, by
using estimates of mean exit times of $X_{\cdot}^{\varepsilon }$ with respect to small neighbourhoods of
isolated singular points of $A$, that $X_{\cdot }^{\varepsilon }$ weakly
converges to a process $X_{\cdot }$, whose law is given by a linear convex
combination of two Dirac measures. Other results in the one-dimensional
case, which rest on large deviation techniques and which cover the example (%
\ref{sgn2}), can be found in \cite{GHR}, \cite{H}. We also mention the paper \cite{BS}, where the authors use viscosity solutions of perturbed
backward Kolmogorov equations to study the small noise problem (\ref%
{SmallNoise}). As for local time techniques we refer to \cite{Tr}. Another
more recent paper, which deals with vector fields $A$ of the form $%
c_{+}\left\vert x\right\vert ^{\beta },x>0$ and $-c_{-}\left\vert
x\right\vert ^{\beta },x<0$ and self-similar driving noise $B_{\alpha }$
with self-similarity index $\alpha >0$, is \cite{PP_self_similar}. Using a new
qualitative technique, which doesn't rely on quantitative tools based e.g.
on Kolmogorov type equations, the authors characterize zero-noise limits by
analyzing solutions to a ``canonical equation", that is solutions to a
transformation of the small noise perturbed equation for $t\longrightarrow
\infty $. Based on the latter time-space rescaling technique the authors in 
\cite{PaP} extend the results in \cite{PP_self_similar} to the case of SDE's with multiplicative noise driven by a strictly $%
\alpha -$stable process and vector fields $A$, which are given by the following
rather large class of coefficients%
\begin{equation*}
A(x)=\left\{ 
\begin{array}{cc}
x^{\beta }L_{+}(x), & x>0 \\ 
-\left\vert x\right\vert ^{\beta }L_{-}(x), & x>0%
\end{array}%
\right. 
\end{equation*}%
for continuous functions $L_{\pm }:(0,\infty )\longrightarrow (0,\infty )$
with $L_{\pm }(x)\sim A_{\pm }l(\frac{1}{x})$ as $x\longrightarrow +0$,
where $l>0$ is a slowly varying function at infinity and $A_{+},A_{-}>0$ are
constants.   

Further, we also point out the works of \cite{AF}, \cite{Att}, which are
concerned with the analysis of small noise limits in the case of
one-dimensional linear transport equations. See also \cite{DLN}, \cite{Ma}
in the case of non-linear PDE's.

\bigskip 

Compared to the above results in the one-dimensional case, however, the
literature on small noise analysis for multi-dimensional vector fields $A$
is rather sparse.

Using viscosity solutions of Hamilton-Jacobi-Bellman equations, Zhang \cite%
{Zh} e.g. discusses a characterization of weak limits of $X_{\cdot
}^{\varepsilon }$ of (\ref{SmallNoise}), when $A:\mathbb{R}^{d}$ $%
\longrightarrow \mathbb{R}^{d}$ is continuous and bounded, and generalizes
results in \cite{BB} in certain cases.

\bigskip As for the case of discontinuous vector fields $A$ in the
multi-dimensional setting we mention the articles Delarue, Flandoli,
Vincenzi \cite{DFV} and  Buckdahn, Ouknine, Quincampoix \ \cite{BOQ}. Based
on a concept of solutions in the sense of Filippov, which differs from the
classical one in the case of discontinuous vector fields, the work \cite{BOQ}
is concerned with study of ODE's (\ref{CP}), when $A$ is merely measurable.
In \cite{DFV} the authors derive probability estimates of exit times with
respect zero-noise limits of certain ODE's in $\mathbb{R}^{4}$ and apply
those to the analysis of small noise perturbations of the Vlasov-Poisson
equation. Another result in this direction, whose proof compared to \cite{BB}
does not presume knowledge of the explicit distribution of $X_{\cdot
}^{\varepsilon },$ is the paper \cite{PP1}. Here the
authors examine small noise limits of time-inhomogeneous vector fields $A$,
which allow for discontinuity points in a halfplane. In \cite{KuP} Pilipenko, Kulik further
 develop the methods of \cite{PP_self_similar, PP1}
and apply new ideas based on an averaging principle, which is closely related to that
for Markov processes, to describe zero-noise limits, when the Wiener noise
driven SDE admits for multiplicative noise and vector fields $A$, that are
locally Lipschitz continuous outside a fixed hyperplane $H$ and have H%
\"{o}lder asymptotics at $H$.   

Finally, we point out
the recent work \cite{DM}. Here the authors establish results on the limiting behaviour of $%
X_{\cdot }^{\varepsilon }$ in (\ref{SmallNoise}) for vector fields $A=\nabla
V$, which are not assumed to be Lipschitz continuous in the neighbourhood of
the origin and where the potential $V$ is $C^{1,1}$ on compact subsets in $%
\mathbb{R}^{d}\backslash \{0\}$ given by functions of the form%
\begin{equation*}
V(x)=g(x)\left\vert x\right\vert ^{1+\alpha }+h(x)\left\vert x\right\vert
^{1+\beta },x\neq 0,V(0)=0
\end{equation*}%
for functions $g,h$ such that $\left\{ g\neq 0\right\} \cap \left\{
h\neq 0\right\} =\emptyset $. There the analysis of zero-noise limits, which
requires rather complicated conditions on $V$, rests on probability
estimates with respect to the polynomial growth of $V(X_{t}^{\varepsilon })$
for $\varepsilon \searrow 0$ and the submartingale dynamics of $%
V(X_{t}^{\varepsilon }),t\geq 0$. 

\subsection{Goals of the paper}

The main objective of our paper is characterization of limits of $X_{\cdot
}^{\varepsilon }$ of small $\alpha -$stable noise $B_{\alpha }$ perturbed
ODE's for $\alpha \in (1,2]$, when $A$ is locally Lipschitz continuous on $%
\mathbb{R}^{d}\backslash \{0\}$ and has the asymptotics $A(x)\sim \overline{a%
}(\frac{x}{\left\vert x\right\vert })\left\vert x\right\vert ^{\beta -1}x$
for $x\longrightarrow 0$, where $\overline{a}$ is a positive continuous
function on the unit sphere and $\beta<1$. 

Let us  set  the problem more precisely.

Let $\{B_\alpha(t), \ t\geq 0\}$ be a $d$-dimensional symmetric L\'evy $\alpha$-stable
process with $\alpha\in(1,2],$ i.e.,
$\{B_\alpha(t), \ t\geq 0\}$ is a c\`adl\`ag process with independent and homogeneous increments
such that the characteristic function of the increment equals
$$
\E \exp\{ i (z,B_\alpha(t))\}=\exp\{-c|z|^\alpha t\}, \ z\in \mbR^d, \ t\geq 0,
$$
where $c>0$ is a constant.
For example, if $\alpha=2,$ then $\{B_\alpha(t), \ t\geq 0\}$ is a Brownian motion.

Let $A:\mbR^d\to \mbR^d$. Consider an ODE
\bel{eq:unperturbed}
d X_0(t) = A(X_0(t)) dt, \ t\geq 0,\ \ \
\ee
\bel{eq:small_stable_lim1}
X_0(0)=0
\ee
and a sequence of its perturbations by a small noise
\bel{eq:perturbed}
d X_\ve(t) = A(X_\ve(t)) dt + \ve d B_\alpha(t), \ t\geq 0,
\ee
\bel{eq:perturbed_ini}
X_\ve(0)=0.
\ee

The problem is to study existence and identification of  a limit $\lim_{\ve\to 0}X_\ve$. 

If the function $A$ satisfies the Lipschitz condition, then it is easy to show the convergence
$X_\ve\to X_0, \ \ve\to0$ for a.a. $\omega$ uniformly on the compact sets. However, 
the problem is much harder if there are multiple solutions to \eqref{eq:unperturbed}.
 This may happen, as mentioned before, if e.g. the Lipschitz property for $A$ fails.

We will use the following notations
$$
r=|x|\geq 0, \ \vf=\frac{x}{|x|}\in S^{d-1},\ x^\beta:=|x|^{\beta-1}x=r^\beta \vf,\ \  x\neq 0.
$$
  If $x=0,$ then the corresponding $\vf$ may be arbitrary; by definition we set $0^\beta:=0$
for arbitrary $\beta.$

A problem solved in this paper is to identify a limit of $\{X_\ve\}$ as $\ve\to0$ if 
\be\label{eq:A_asympt}
A(x)\sim \bar a(\vf)x^{\beta},\ \ x\to0,
\ee
  where $\beta<1$, 
$\bar a: S^{d-1}\to (0,\infty)$
 is a positive continuous function on the sphere, and $A$ is locally Lipschitz function 
 in $ \mbR^d\setminus \{0\}.$

If $\alpha+\beta >1$ and  $\alpha \in (1,2)$, then the existence of a  weak solution  follows from \cite{Portenko1995, Bogdan},
 uniqueness was  proved in \cite{ChenWang}.
 Case $\alpha=2$ means that the noise is a Brownian motion, existence of uniqueness and uniqueness for the case $\beta>-1$ is well known.
 If $d=1$ and  $\alpha+\beta<1 $, then  uniqueness  fails, see the
arguments of  \cite[Theorem 3.2]{TanakaTsuchiya74}. 

It can be easily seen that any limit point $X_0$ of 
$\{X_\ve\}$ (if it exists)
 satisfies ODE \eqref{eq:unperturbed} outside of 0. Since $\bar a(x)>0$, equation
 \eqref{eq:A_asympt} yields that
 $X_0$ does not return to 0 after the exit from 0. Hence, to  identify  $X_0$ we have 
 to study two problems:  find the distribution of time spent at 0 and  characterize ``direction
 of exit'' from 0.

We will show that $X_0$ spends zero time at 0 and the distribution of 
``direction of exit'' coincides with  
 the distribution  of ``direction of exit'' of the limit for the sequence of the following model equations 
\be\label{eq_Levy_eps}
d \bar X_\ve(t) = \bar a(\bar \vf_\ve(t)) \bar X^\beta_\ve(t) dt + \ve d B_\alpha(t), \ t\geq 0,
\ee
$$
\bar X_\ve(0)=0,
$$
where $\bar \vf_\ve(t):=\frac{\bar X_\ve(t)}{|\bar X_\ve(t)|}$. This is quite natural because a selection   of exit should be done in a neighborhood of 0;
 coefficients of the initial equation \eqref{eq:perturbed} and the model equation \eqref{eq_Levy_eps}
 are equivalent there.

Moreover, we will show that the exit distribution equals the distribution of
the limit as $t\to\infty$ of  
angle $ \frac{\bar X_\ve(t)}{|\bar X_\ve(t)|}$  for the model equation.
 It will also be shown that this limit exists a.s. and its distribution is independent of $\ve>0.$

To understand what happens with the limit of \eqref{eq:perturbed}, let us firstly explain what happens
with 
the limit of \eqref{eq_Levy_eps}.
 
The formal limit equation for
\eqref{eq_Levy_eps} is
\be\label{eq_Levy_limit}
\begin{cases}
d \bar X_0(t) = \bar  a(\bar \vf_0(t)) \bar X_0^\beta(t) dt, \ t\geq 0,\\
\bar X_0(0)=0.
\end{cases}
\ee
Any solution to \eqref{eq_Levy_limit} is of the form:
\be\label{eq_sol_ODE1}
 \bar X_0(t)= \bar X_0(t; t_0, \vf):=\begin{cases}
  \big(\bar a(\vf )(1-\beta) (t-t_0)\big)^{\frac{1}{1-\beta}}\vf ,&  t\geq t_0,\\
 0, & t\in [0,t_0],
\end{cases}
\ee
where $\vf\in S^{d-1}, t_0\in[0,\infty]$ are constants.

That is,
\be\label{eq_sol_ODE}
\bar \vf_0(t)\equiv \vf, \ \ \ \mbox{ and }\ \ \ \bar r_0(t)=
  \begin{cases}
  \big(\bar a( \vf )(1-\beta) (t-t_0)\big)^{\frac{1}{1-\beta}} ,&  t\geq t_0,\\
 0, & t\in [0,t_0].
\end{cases}
\ee
 
It can be seen from the change of variables formula and the self-similarity of
 $B_\alpha$ that
\bel{eq:transform_model}
\{\bar X_\ve(t), \ t\geq 0\} \overset{d}=
\{  \ve^{\frac{\alpha}{ \alpha+\beta-1}}\bar X_1(\ve^{\frac{\alpha(\beta-1)}{ \alpha+\beta-1}}t),
\ t\geq 0\}.
\ee
Hence the distribution of the limit angle  
 $  \bar \Phi_\ve(t)=  \frac{\bar X_{\ve}(t)}{|\bar X_{\ve}(t)|}$ as $\ve\to0$
   should coincide  with
the distribution of  
\bel{eq:236}
\bar \Phi(+\infty):= \lim_{t\to\infty}\bar \Phi_{\ve_0}(t)
\ee
 for any fixed $\ve_0>0.$
A.s. existence of the corresponding limit is given in \S\ref{section:smallNoise},
see details in \S\ref{subsect:proof of Th4_2}.

As we mentioned above, we will show that any limit process spends zero time at 0. Hence 
\[
\bar X_\ve(\cdot)
\Rightarrow \bar X_0(\cdot; 0, \bar \Phi(+\infty)),\ \ve\to0.
\]

 The paper is organized as follows.
In \S \ref{section:large_time} we study  a  general problem of finding the asymptotic
behavior as $t\to\infty$ of a solution to the SDE
\bel{eq:1.1.3}
dX(t) =A(X(t))dt + \ general \ noise.
\ee
This question is interesting by itself, see for example \cite{GS, GO, KKR, BKT, PP_self_similar, Yuskovych}.
We give sufficient conditions that ensure existence of a.s. limit of the angle
 $\Phi(\infty):=\lim_{t\to\infty}\frac{X(t)}{|X(t)|}$ if $A(x)\sim \bar a(\vf)x^\beta, \ |x|\to\infty$.
In particular, our result  will imply existence of the limit in \eqref{eq:236}.
It will also be shown that the absolute value $R(t)=|X(t)|$ is equivalent 
to $\big(\bar a( \Phi(\infty) )(1-\beta) t\big)^{\frac{1}{1-\beta}} 
$, 
 compare with formula \eqref{eq_sol_ODE1}.
 
The main result of the paper is formulated and proved
 in section \ref{section:smallNoise}. It states that under some natural assumptions
we have the following convergence:
\bel{eq:result}
X_\ve(\cdot)\Rightarrow X_0(\cdot, 
\bar \Phi(+\infty)),\ \ \ve\to0,
\ee 
where $\bar \Phi(+\infty)
$ is defined for the model equation in \eqref{eq:236} and $X_0(t,\vf)$ is a solution of the limit equation 
\eqref{eq:unperturbed} such that $X_0(t,\vf)\neq0, t>0$ and
 $\lim_{t\to0+}\frac{X_0(t,\vf)}{|X_0(t,\vf)|}=\vf.$
 
Proofs of auxiliary results are given in the Appendix. We also give there 
sufficient conditions that ensure existence of a family of solutions $X_0(t,\vf)$
to the limit equation \eqref{eq:unperturbed} such that $\lim_{t\to0+}\frac{X_0(t,\vf)}{|X_0(t,\vf)|}=\vf.$

\begin{remk}
    The study of the distribution of the limit angle  angle   $\bar \Phi(+\infty)$ is an interesting question. We conjecture that under general assumptions about the function $\bar a$ it is neither uniform nor discrete.  Let us explain (non-rigorous) arguments in favour of our hypothesis in the two-dimensional case when the noise is a Brownian motion ($\alpha$=2). Recall that the distribution of $\bar \Phi(+\infty)$ is independent of $\ve,$ so let's take $\ve=1.$  In this case, after applying It\^o's formula,  identifying $\phi$ with $(\cos \phi, \sin \phi)$, and standard arguments we obtain the following equation for absolute value and polar angle of $\bar X_1:$
\[
 d\bar R(t)=\left(\bar a(\bar \Phi(t)) \bar R(t)^\beta + \frac{1}{2\bar R(t) }\right)dt+ dw_1(t),
\]
\[
d\bar\Phi(t) =\frac{dw_2(t)}{\bar R(t)},
\]
where $w_1, w_2$ are two independent one-dimensional Wiener processes. 

So, for any $t_0>0$ and any $t\geq t_0:$
\[
\bar \Phi(t)=\bar \Phi(t_0)+\int_{t_0}^t\frac{dw_2(s)}{\bar R(s)}=\Phi(t_0)+ w\left(\int_{t_0}^{t }  {\bar R^{-2}(s)} {ds}\right),\]
where $w$ is a Wiener process. Hence, $\bar \Phi(+\infty) =\Phi(t_0)+ w\left(\int_{t_0}^{\infty } {\bar R^{-2}(s)} {ds}\right)$. Recall that under our assumptions, the integral $\int_{t_0}^{\infty } {\bar R^{-2}(s)} {ds}$ converges with probability 1. Note that the larger the function $\bar a$ is, the faster $\bar R$ grows, and the correspondingly slower the angle $\bar \Phi$ oscillates. Therefore, it is more likely that the limiting value of the angle will be in regions where the function $\bar a$ is larger. 

  Note also that if we take 
 $X_\ve(0)=x\ve^{\frac{\alpha}{ \alpha+\beta-1}} $ instead of $X_\ve(0)=0,$
 then the distribution of the limit angle may be different, see equality \eqref{eq:transform_model} for the model equation. 
\end{remk}

\section{An asymptotic behavior of a solution to an SDE for  a large time}\label{section:large_time}

Let $X(t)$ be a solution to the following SDE in $\mbR^d$:
\be\label{eq:multidimSDE}
\ba
dX(t) =A(X(t))dt + B(X(t))dW(t) + \int_U C_1(X(t-),u)  N_1(du, dt)+\\
+\int_U C_2(X(t-),u)  \hat N_2(du, dt),
\ea
\ee
where $W$ is a multi-dimensional Wiener process, $N_1$ is a Poissonian point measure,
$\hat N_2$ is a compensated Poissonian point measure.

In this section we will always assume that any equation has a (weak) solution for any starting point, any integral is well-defined, etc.

Together with \eqref{eq:multidimSDE}, let us consider an ODE
\be\label{eq:ODE}
d\bar X(t)= \bar A(\bar X(t)) dt,
\ee
 where $\bar A$ is such that $\bar A(x)\sim A(x)$ as $|x|\to\infty$, i.e. $|\bar A(x)- A(x)|= o(|A(x)|)$ as $|x|\to\infty.$

The aim of this section is to find conditions that ensure equivalence
\[
X(t)\sim \bar X(t), \ t\to\infty, \ \ \mbox{a.s.}
\]
Note that the study of asymptotics for a solution to the ODE
 \eqref{eq:ODE} may be much simpler
 than   for the SDE \eqref{eq:multidimSDE}.

Naturally we should have some assumptions that guarantee that the rate of growth of the noise terms
in \eqref{eq:multidimSDE} is $o(|\bar X(t)|)$ as $t\to\infty $ a.s. We also have to discuss an initial condition or another
parametrization for $\bar X(t)$.

Consider an integral equation
\be\label{eq:integral_eq}
dZ(t) =A(Z(t)) dt +d\xi(t),
\ee
where $\xi$ is locally bounded, measurable  function, $\xi(0)=0.$

At first let us obtain  a result on the asymptotics for equation \eqref{eq:integral_eq} with non-random
 $\xi$ and then we apply the results to equation \eqref{eq:multidimSDE} with
\[
d\xi(t)=
B(X(t))dW(t) + \int_U C_1(X(t-),u)  N_1(du, dt)
+\int_U C_2(X(t-),u)  \hat N_2(du, dt).
\]
We will consider the case when $A$ has power asymptotics
at the infinity with an index that is less than 1. 

\begin{remk}
If the coefficients have a power growth with  an index that is less than 1, then solution will
have the power asymptotic. However if coefficients has a linear growth, then
the growth of a solution may be exponential.  This case demands another technique and we do not consider
this case in the paper.
\end{remk}
Recall notations
\[
r=|x|,\ \vf=\frac{x}{|x|},\ r(t)=r_X(t)=|X(t)|, \ \vf(t)=\vf_X(t)=\frac{X(t)}{|X(t)|}.
\]
We define $\vf$ arbitrary if denominator equals 0.

If we specify an initial condition  $X(0)=x, $ then we denote $X(t)$ by $X_x(t)$.

Represent the vector field $A$ as a sum of the radial and tangential components:
\bel{eq:A_representation}
A(x)=A_{rad}(x) +A_{tan}(x),
\ee
i.e.,
\[
A_{rad}(x) = \lg A(x), \vf\rg \vf, \ A_{tan}(x)=A(x)- \lg A(x), \vf\rg \vf.
\]

\begin{thm}\label{thm:integral_eq_growth}
Assume that
\begin{itemize}
\item
\bel{eq:A_r}
A_{rad}(x)=a(x)r^\beta\vf
\ee
and
\bel{eq:298} 
\forall \vf_0\in S^{d-1}\ \  \ \exists \lim\limits_{\substack{r\to+\infty\\ \vf \to\vf_0}} a(x)
=: \bar a(\vf_0),
\ee
where  $\beta<1$ and $\bar a: S^{d-1}\to \mbR$ is a positive and continuous function;
\item
\bel{eq:A_t_infty}
\exists \gamma>0\ \exists C_{tan}>0\ \exists R_{tan}>0 \
\forall  \ r>R_{tan}\ : \ \  \sup_{|y|= r} |A_{tan}(y)|  \leq C_{tan}r^{\beta-\gamma};
\ee
\item
$\exists   \delta>0\ \exists C_\xi>0\ \forall t\geq 0:$
\bel{eq:estim_noise} 
|\xi(t)| \leq C_\xi\left(1+ t^{\frac{1}{1-\beta}-\delta}\right).
\ee
\end{itemize}
Then there is $R_0=R_0(A, \delta, \beta, C_\xi)>0$ such that for any
 $x\in \mbR^d,\ |x|\geq R_0$ and any solution 
 $Z(t)=Z_x(t)$, $Z_x(0)=x$ of \eqref{eq:integral_eq}  we have:
\begin{itemize}
\item  $\lim_{t\to\infty}|Z_x(t)|=+\infty$;
\item  there exists a limit 
\[
\vf_{Z_x}(+\infty):=\lim_{t\to\infty}\vf_{Z_x}(t):=
\lim_{t\to\infty}\frac{Z_x(t)}{|Z_x(t)|};
\]
\item $Z_x(t) \sim \left( (1-\beta) \bar a(\vf_{Z_x}(+\infty)) t  \right)^{\frac{1}{1-\beta}}\vf_{Z_x}(+\infty)$ as $t\to\infty.$
\end{itemize}
\end{thm}
We shall postpone the proof of the theorem to the Appendix, and now make a few remarks on the assumptions of the theorem.
\begin{remk}
We don't assume uniqueness of a solution. The limit angle $\vf_{Z_x}(+\infty)$
may depend on a choice of a solution $Z_x(t)$.
\end{remk}
 \begin{remk}
Recall that the function
\bel{eq:funcY}
\bar X(t)=\bar X(t,\vf):=\left( (1-\beta) \bar a(\vf) t  \right)^{\frac{1}{1-\beta}}\vf
\ee
 is a solution to
\eqref{eq_Levy_limit}.
The statement of   Theorem
\ref{thm:integral_eq_growth}  means that
 $Z_x(t)\sim \bar X(t,\vf_{Z_x}(+\infty))$ as $t\to\infty$.
Note that \eqref{eq:estim_noise} yields $ |\xi(t)|=o(|\bar X(t)|), \ t\to\infty,
 $ 
 i.e., the perturbation is negligible w.r.t. the solution of \eqref{eq_Levy_limit}. If the noise is large in the sense:
 \[
\limsup_{t\to\infty}\frac{|\xi(t)|}{t^{\frac{1}{1-\beta}+\delta}}>0
\]
for some $\delta>0,$ then there are no reasons to expect equivalence $Z_x(t)\sim \bar X(t,\vf)$ as $t\to\infty$ for any $x$ and $\vf.$

Existence of $\delta>0$ such that $|\xi(t)| \leq C_\xi\left(1+ t^{\frac{1}{1-\beta}-\delta}\right) $
is equivalent to existence of $\alpha>0$ such that $\alpha+\beta>1$ and 
\bel{eq:698}
|\xi(t)| \leq C_\xi\left(1+ t^{\frac{1}{\alpha}}\right).
\ee
The condition $\alpha+\beta>1$ appears often in theorems on existence and uniqueness of solutions for SDEs with $\alpha$-stable noise. Moreover this noise satisfies  condition \eqref{eq:698}, see Example \ref{expl:stable}.\end{remk}
\begin{remk}
Unless we assume that the tangential component is negligibly small with respect to the radial component, the angular part may not stabilise as $t\to\infty$.
 Indeed, consider two-dimensional case,
where 
\[
\xi(t)=0, t\geq 0,\ \ \ \ A_{rad}(x)=x^\beta= (r^{\beta-1}x_1, r^{\beta-1}x_2) \  \mbox{and} \ A_{tan}(x)=  (-r^{\beta-1} x_2, r^{\beta-1} x_1).\]
Then $d r(t)=r^\beta(t) dt$ and 
\[
r(t)=\big(r(0)^{1-\beta} +{(1-\beta)}t\big)^{\frac{1}{ 1-\beta}}. 
\]
If we identify the angular part $ (\frac{x_1}{|x|}, \frac{x_2}{|x|})=:(\cos \phi, \sin \phi)$ with the polar angle $\phi\in \mbR$, then 
\[  \phi'(t)=r^{\beta-1}(t) = \big(r(0)^{1-\beta} +{(1-\beta)}t\big)^{-1}, t\geq 0.\]
So, $\phi(t)\sim (1-\beta)^{-1} \log t, \ t\to\infty.$

The next example shows that if the tangential part is large enough, then assumptions \eqref{eq:estim_noise},  \eqref{eq:298}, and  \eqref{eq:A_r}  (without \eqref{eq:A_t_infty}) don't even
 imply convergence $\lim_{t\to\infty}|Z_{x_0}(t)|=+\infty$ for some   $x_0$ that may be selected arbitrary large.
Consider the case $d=2$ and assume that $A_{rad}(x)=x^\beta= r^\beta \vf, |x|\geq R_{rad}$. 
  For any starting point $x=(n,0), n\in\mbN,$ we shall provide a coefficient $A_{tan}$ and   the noise term  $\xi$ such that $|\xi(t)|\leq 1, t\geq0$ and the function $Z_{x_0}$ 
is bounded.

Let $n\in\mbN$ be   fixed, $r(t)=r_n(t)$ be the solution to the ODE
\[
\frac{dr(t)}{dt}= r^\beta(t),\ \ r(0)=n,
\]
i.e., $r(t) =\left(n^{1-\beta}+(1-\beta)t\right)^{\frac{1}{1-\beta}}.$

Denote $\sigma:=\frac{(n+1)^{1-\beta}-n^{1-\beta}}{1-\beta}.$ Then $r(\sigma)=n+1.$

{Let $\wt Z_{(u,v)}(t), u,v\in \mbR, t\geq 0$
be a solution to the ODE
\[
d \wt Z_{(u,v)}(t)= A(\wt Z_{(u,v)}(t)) dt,\ \ \wt Z_{(u,v)}(0)=(u,v).
\]
It is possible to select  $A_{tan}$ such that $\wt Z_{(n,0)}(\sigma)= (-n-1,0)$ and $\wt Z_{(-n,0)}(\sigma)= (n+1,0) $ for all $n\in\mbN.$
Let $Z(0)=(n,0) $, 
\[
\xi(t):=
\begin{cases}
(0,0),\ t\in[2k\sigma, (2k+1) \sigma);\\
(1,0), \ t\in [(2k+1)\sigma, (2k+2) \sigma).
\end{cases}
\]
Then 
\[
Z(t):=
\begin{cases}
\wt Z_{(n,0)}(t-2k\sigma),\ t\in[2k\sigma, (2k+1) \sigma);\\
\wt Z_{(-n,0)}(t-(2k+1)\sigma), \ t\in [(2k+1)\sigma, (2k+2) \sigma).
\end{cases}
\]
 for any $k\geq 0. $ Hence the function $Z_{(n,0)}(t), t\geq 0$ is bounded.
 }
\end{remk}

\begin{remk}
Notice that in Theorem \ref{thm:integral_eq_growth} we assume some asymptotics for coefficients
at infinity but for equation 
\eqref{eq:perturbed} we  assume similar asymptotics at zero, see  \eqref{eq:A_asympt}. 
We will see that   the limit as $\ve\to0$ of  time-space transformation    $
\tilde X_\ve(t):=\ve^{\frac{-\alpha}{  \alpha+\beta-1}} 
  X_\ve (\ve^{\frac{\alpha(1-\beta)}{ \alpha+\beta-1}}t),
 $
  c.f.
\eqref{eq:transform_model}, equals a solution to the model equation  \eqref{eq_Levy_eps} with $\ve=1$.
This is the reason that the study of
the limit angle
 at 0 for equation \eqref{eq:perturbed}  with the small noise reduces to  the
 study of the limit angle of the model equation 
 for large $t.$ 
\end{remk}

The main result of this section is the following.

\begin{thm}\label{thm:SDE_growth}
Suppose that for any initial condition $x\in\mbR^d$
there is a unique solution to \eqref{eq:multidimSDE}, which is a strong Markov process.

Assume that \eqref{eq:A_r},  \eqref{eq:A_t_infty} are satisfied, and
\bel{eq:X_infty}
\forall x\ \ \  \ \Pb_x\left( \sup_{t\geq 0}|X(t)|=+\infty\right) =1,
\ee
\bel{eq:cond_noise}
\exists \delta>0\ \forall \ve>0 \ \exists c=c_{\delta, \ve} \ \forall x \ \ \
 \Pb_x\left( |\xi_X(t)|\leq c(1+ t^{\frac{1}{1-\beta}-\delta}), \ t\geq 0\right) \geq 1-\ve,
\ee
where
\[
\xi_X(t)=
\int_0^tB(X(s))dW(s) + \int_0^t\int_U C_1(X(s-),u)  N_1(du, ds)
+\int_0^t\int_U C_2(X(s-),u)  \hat N_2(du, ds).
\]

Then   for any initial condition $X(0)$ we have
\begin{itemize}
\item
\bel{eq:unboundX}
\lim_{t\to\infty}|X(t)|=+\infty\ \ \mbox{a.s.};
\ee
\item  there exists a (random) limit
\[
\vf_{X}(+\infty):=\lim_{t\to\infty}\vf_{X}(t)=
\lim_{t\to\infty}\frac{X(t)}{|X(t)|} \ \ \ \mbox{a.s.};
\]
\item $X(t) \sim \bar X(t, \vf_{X}(+\infty))$
 as $t\to\infty$ a.s.,
 where $\bar X$ is defined in \eqref{eq:funcY}.
\end{itemize}
\end{thm}
\begin{proof}
Set
\[
  \tau_{R}:=\inf\{ t\geq 0 \ : \ |X(t)| \geq R\}.
\]
 It follows from the assumptions of the Theorem that $\tau_R<\infty$ a.s. for any $R>0 $ and any
 initial starting point $X(0)$.

 Let $\ve>0$ be arbitrary, constants $\delta>0$ and $c=c_{\delta,\ve}$  be from
 \eqref{eq:cond_noise}, $R_0=R_0(\delta, c, A)$ is selected from Theorem \ref{thm:integral_eq_growth}.

 Theorem \ref{thm:integral_eq_growth} and the strong Markov property yield
 \bel{eq:407}
 \ba
 \Pb(\lim_{t\to\infty}|X(t)|=+\infty)\geq \Pb\left(\sup_{t\geq \tau_{R_0}} |\xi_X(t)-\xi_X(\tau_{R_0})|
\leq c\left(1+ (t- \tau_{R_0})^{\frac{1}{1-\beta}-\delta}\right)\right)=
  \\
 \int_{|x|\geq R_0}
 \Pb_x\left(\sup_{s\geq 0} |\xi_X(s)|\leq c(1+ s^{\frac{1}{1-\beta}-\delta}) \right) \Pb_{X(\tau_{R_0})}(dx)\geq 1-\ve.
 \ea
 \ee
 Since $\ve>0$ is arbitrary, the last inequality implies
\[
\lim_{t\to\infty}|X(t)|=+\infty\ \ \mbox{a.s.}
\]
Proofs of  all other items are similar.
\end{proof}
\begin{remk}
We need the strong Markov property only for the  justification of the equality in \eqref{eq:407}. Generally,
we may replace assumption on Markovianity, \eqref{eq:X_infty} and \eqref{eq:cond_noise}  with
\[
\ba
 & \exists \delta>0, c>0 \ \forall R>0 :\\
& \Pb\left(\exists t_0 \forall  t\geq 0\ \  |\xi_X(t_0+t)-\xi_X(t_0)|\leq c(1+ t^{\frac{1}{1-\beta}-\delta})\ \ \mbox{and}\ \ |X(t_0)|\geq R\right) =1.
\ea
\]
\end{remk} 

\begin{expl}\label{Expl:exit}
Condition \eqref{eq:X_infty}  is satisfied if, for example, all coefficients  are locally bounded
and  $B$ satisfies ellipticity condition or, for example, if the jump part is non-degenerate at infinity and $X$
can exit from any fixed ball by one large jump with positive probability that is independent of the initial point (but may depend on a ball):
\[
\forall R>0: \inf_{x}\Big(\int_{u: |C_1(x,u)|\geq  R}\nu_1(du)+  \int_{u: |C_2(x,u)|\geq  R}\nu_2(du)\Big)>0.
\]
Here $N_1, N_2$ are Poisson point measures with intensities  $dt\times \nu_1(du), dt\times \nu_2(du)$, respectively.
\end{expl}
Here are three examples where the apriori bound \eqref{eq:cond_noise} of noise growth is satisfied.
\begin{expl}\label{expl:stable}
Let $X(t), t\geq 0$ be a solution of the SDE
\[
dX(t) =A(X(t)) dt+ dB_\alpha(t), t\geq 0,
\]
where $\alpha>1$, $A$ is locally bounded and satisfies \eqref{eq:A_r},  \eqref{eq:A_t_infty}. Then all conditions
  of Theorem \ref{thm:SDE_growth} are satisfied. Indeed,  
existence and uniqueness of the solution follows from
\cite{Portenko1995, Bogdan} and \cite{ChenWang} if  $ \alpha+\beta>1$.
Condition  \eqref{eq:X_infty}   is fulfilled since $B_\alpha$ has any large jumps a.s.  Condition  \eqref{eq:cond_noise} of noise growth  
follows from 
\bel{eq:Khintchine}
\forall \alpha'<\alpha\ : \  \ \ \lim_{t\to\infty}\frac{B_\alpha(t)}{t^{\frac1{\alpha'}}}=0 \ \ \mbox{a.s.},
\ee
see \cite{Khintchine}.
\end{expl}
\begin{expl}
Let $X(t)$ be a solution to \eqref{eq:multidimSDE}, where   $C_1=C_2\equiv 0$ and $B$ is bounded.
Observe that
\[
\sum_j\int_0^t B_{kj}(X(s)) dW_j(s) = \tilde W_k\left(
\sum_j B^2_{kj}(X(s)) ds,
\right)
\]
where $\tilde W_k$ is a Brownian motion.

 Hence $|\xi_X(t)|$ is dominated by
$\sup_{s\in[0,t]}(\sum_k|\tilde W_k(cs)|^2)^{1/2}$, where 
\[
c=\sup_x\sum_{k,j}|B_{k,j}(x)|^2=\const.
\]

It is well known that
\[
\forall \gamma>1/2:\ \ \lim_{t\to\infty}\frac{\tilde W_k(t)}{t^\gamma}=0\ \ \ \mbox{a.s.}
\]
So, condition \eqref{eq:cond_noise} is satisfied if  $\beta>-1.$

It should be noted that the case $\beta=-1$ is really critical. Indeed, if $d=1$, then the Bessel process
\[
d\zeta(t)=\frac{a}{\zeta(t)} dt +dW(t)
\]
does not converge to $\infty$ as $t\to\infty$ for $a\leq 1/2.$
\end{expl}
\begin{expl}
Assume that the noise term is of the form
\[
\xi_X(t)= \int_0^t C(X(s-)) d \tilde Z(s),
\]
where $\tilde Z$ is a zero mean L\'evy process without a Gaussian component.
Assume that its jump measure $\nu$
satisfies conditions
\[
\int_{|z|> x}\nu(d z) \leq \frac{K}{x^\alpha}, \quad x\geq 1,
\]
and
\[
\int_{|z|\leq 1} z^2\nu(d z) \leq K
\]
for some   $K>0$ and $ \alpha\in (1,2)$.

It can be proved, see \cite[Lemma 3.1]{PaP}  that for any $\ve>0$ and $\alpha'<\alpha$ there exists a 
generic constant $C=C(K,\alpha, \alpha',\theta)$
such that for any predictable process $\{\sigma(t)\}_{t\geq 0}$, $|\sigma(t)|\leq 1$ a.s., we have
\[
\Pb \Big(|\int_0^t  \sigma(s)\, d \tilde Z(s)|\leq C (1+t^{\frac{1}{\alpha'}}),\ t\geq 0\Big)
\geq 1-\ve.
\]
Hence, condition   \eqref{eq:cond_noise} is satisfied if  $\alpha+\beta>1,$ cf. with Example \ref{expl:stable}. 
\end{expl}


\section{Small-noise perturbations of ODEs}\label{section:smallNoise}

Let $\{X_\ve(t)\}$ be a solution to the SDE \eqref{eq:perturbed} with the initial condition
$
X_\ve(0)=0$.

In this section we formulate and prove the main result of the 
paper on a limit of $\{X_\ve\}$ as $\ve\to0,$ see Theorem \ref{thm:main_limit_small_noise} below.

 


\begin{assum}\label{assum:0} The function $A$ is such that for any $\ve>0$ and any initial starting point there exists a unique weak solution of
\eqref{eq:perturbed}, and the solution is a strong Markov process.
\end{assum}

 Assumption \ref{assum:0}  holds  true if, for example,  $A\in L_{p, loc},$ where $p>\frac{d}{\alpha-1}$ and $\alpha>1,$
see existence in \cite{Portenko1995, Bogdan} and uniqueness in \cite{ChenWang}.


 


Recall that radial and tangential components $A_{rad}$ and $A_{tan}$  of the vector field $A$ are defined in
  \eqref{eq:A_representation}.

\begin{assum}\label{assum:1}
\begin{enumerate}
\item the vector field $A$ is  Lipschitz continuous function in any compact
 set $G\subset \mbR^d\setminus\{0\}$;
\item
 we have representation in a neighborhood of zero:
\[
A_{rad}(x)=a(x)r^\beta\vf,
\]
where $ |\beta|<1 $ and there is a positive continuous function $\bar a : S^{d-1}\to \mbR $ such that
\[
\forall \vf_0\in S^{d-1} \ \ \ \exists \lim\limits_{\substack{r\to0+\\ \vf \to\vf_0}} a(x)=: \bar a(\vf_0);
\]

\item
\bel{eq:A_t_0}
\exists \gamma>0\ :\ \sup_{|y|=r}|A_{tan}(y)|=o(r^{\beta+\gamma}),\ r\to0+;
\ee
\end{enumerate}

\end{assum}

We need the following assumption on existence of a collection  of solutions to
 \eqref{eq:unperturbed}, \eqref{eq:small_stable_lim1}, which are parametrized by the angle
$\vf$ of exit from 0.

\begin{assum}\label{assum:2}
For any $\vf\in S^{d-1}$ there is a unique solution $X_0(t)=X_0(t,\vf)$  to
\eqref{eq:unperturbed}, \eqref{eq:small_stable_lim1}
such that $X_0(t)\neq 0, t>0$ and
\[
\lim_{t\to 0+}\frac{X_0(t,\vf)}{|X_0(t,\vf)|}=\vf.
\]
Moreover, for any $T>0$  the pair
$(R_0(t,\vf), \Phi_0(t,\vf)):=(|X_0(t,\vf)|, \frac{X_0(t,\vf)}{|X_0(t,\vf)|})$   
is uniformly continuous for  $(t, \vf)\in (0,T]\times S^{d-1}.$


\end{assum}
  
The function $A$ satisfies Assumption \ref{assum:2} if, for example,
\[
A_{rad}(x)=a(r, \vf) r^\beta,\ \ A_{tan}(x)=b(r, \vf) r^{\beta+\delta }, \ \ r\in(0,1],\ \vf\in S^{d-1},\]
where 
  $\delta\in(0,1), \, \beta\in(-\infty,1)$,   the function $a: [0,\infty)\times\mbR^d\to\mbR$  is positive, continuous, bounded and separated from zero, 
the function $b=b(r,\vf): [0,\infty)\times\mbR^d\to\mbR^d$   is bounded and continuous,  $a$ and $b$  are 
Lipschitz continuous in $\vf$ uniformly in $r.$
Indeed,   $X_0$ is a solution of  \eqref{eq:unperturbed}, \eqref{eq:small_stable_lim1}
such that $X_0(t)\neq 0, t>0$ if and only if its radial and angular parts  satisfy the
equation
\bel{eq:1035} 
\ba
dR_0(t)=a(R_0(t), \Phi_0(t)) R_0^\beta(t) dt, \\
d\Phi_0(t)=(I- \Phi_0(t)\Phi_0^T(t))b(R_0(t), \Phi_0(t)) R_0^{\beta+\delta-1}(t) dt,\\
R_0(t)>0,\ t>0,
\ea
\ee
where $I$ is identity $d\times d$ matrix, $\vf^T$ is transposed vector.
Existence, uniqueness and other properties of  solution of \eqref{eq:1035}  follow from Theorem \ref{thm:Lip_expl}, see Appendix.

\begin{thm}\label{thm:main_limit_small_noise}
Suppose that Assumptions \ref{assum:0}, \ref{assum:1}, \ref{assum:2}  are satisfied and $\alpha\in(1,2],  \ \beta<1,$ and $\alpha+\beta>1.$

Let $ \bar X $ be a solution to the SDE
\bel{eq:SDE_lim0}
d\bar X(t) =\bar a(\frac{\bar X(t)}{|\bar X(t)|})\bar X^\beta(t)dt +dB_\alpha(t), \ t\geq 0,
\ee
\bel{eq:SDE_lim0init}
\bar X(0) =0,
\ee
and
\bel{eq:25.1}
\bar\Phi(+\infty):=\lim_{t\to\infty}\frac{\bar X(t)}{|\bar X(t)|}
\ee
(existence of a.s. limit follows from Theorem \ref{thm:SDE_growth}, see  Example \ref{expl:stable}).

Then  solutions of \eqref{eq:perturbed}, \eqref{eq:perturbed_ini}  converge  in distribution
\[
X_\ve(\cdot) \Rightarrow X_0(\cdot, \bar\Phi(+\infty)),\ \ve\to0,
\]
in the Skorokhod space ${D([0,\infty))}$ of c\`adl\`ag functions, where $X_0$ is from Assumption \ref{assum:2}, $\bar\Phi(+\infty)$ is defined in \eqref{eq:25.1}.
\end{thm}

The course of the proof is the following. 
At first we show   the weak relative compactness 
of $\{X_\ve\}$ and that any limit point is continuous a.s. Then we will show that $X_\ve$
 spends a small time in a small neighborhood of zero. So any limit point spends zero time at 0. 
The noise disappears in the limit as $\ve\to0$, so any limit process $X_0$ must satisfy  
\eqref{eq:unperturbed} for $t$ such that $X_0(t)\neq 0$.  Since $\bar a(x)>0$ 
in Assumption \ref{assum:1}, for any $t_0>0$ such that $X_0(t_0)\neq 0,$ we have
  $X_0(t)\neq 0, t\geq t_0$ a.s. Since $X_0$ spends zero time at 0, we get $X_0(t)\neq 0$ 
 for all $t>0$ a.s. Hence,  any limit point $X_0(t)$ must be of  the form 
$X_0(t,\Phi), $ where $X_0(t,\vf)$ is defined in Assumption \ref{assum:2} and $\Phi=\lim_{t\to0}\frac{X_0(t)}{|X_0(t)|}$.

Denote $\tau^{(\ve)}_\delta:=\inf\{ t\geq 0\ : \  |X_\ve(t)|\geq \delta\}$.
 We shall show that if $\ve>0$ and $\delta>0$ are small, $\ve\ll \delta$, then
\bel{eq:672}
\frac{X_\ve(\tau^{(\ve)}_\delta) }{|X_\ve(\tau^{(\ve)}_\delta)|}\overset{d}\approx \bar\Phi(+\infty)
\ \ \mbox{and} \ \ |X_\ve(\tau^{(\ve)}_\delta)|  \approx \delta,\ \ \mbox{so} \ \ \ \ X_\ve(\tau^{(\ve)}_\delta) \overset{d}\approx \delta  \bar\Phi(+\infty),
\ee
where $\bar\Phi(+\infty)$ is defined in \eqref{eq:25.1}.

Here signs $\approx$ (or $\overset{d}\approx$)  mean that processes are ``close'' (``close'' in distribution).
 The rigorous statements will be given
later in subsection \ref{subsect:proof of Th4_1}.

This is the most difficult part of the proof. To do this, we  apply a space-time transform, see  \eqref{eq:space-time}
below,
and    analyze the behaviour of $X_\ve$ in a microscopic neighborhood of 0.

 By $X_{\ve,y}$, $\ve\geq0$ we denote a solution to   equation \eqref{eq:perturbed}  with initial starting point $X_{\ve,y}(0)=y.$

If we have 
\eqref{eq:672} and we know that $\tau^{(\ve)}_\delta$ is small, then the rest
 of the proof  follows from the next reasoning, where we use Assumption \ref{assum:2}
  combined with  some continuity arguments
\[
X_\ve(t)\approx X_\ve(\tau^{(\ve)}_\delta+t) \overset{d}= X_{\ve, y }(t) |_{y= X_\ve(\tau^{(\ve)}_\delta )}
\overset{d}\approx X_{\ve,y}(t)|_{y= \delta \bar \Phi(+\infty)} \overset{d}\approx X_{0, \delta \bar \Phi(+\infty)}(t)
\approx X_{0}(t,  \bar \Phi(+\infty)).
\]

\subsection{Proof of Theorem \ref{thm:main_limit_small_noise}. First steps}\label{subsect:proof of Th4_2}

\begin{lem}\label{lem:3.5}
 $\{ X_\ve\}$ is weakly relatively compact in $D([0,\infty)).$ Any 
limit process of $\{ X_\ve\}$ as $\ve\to0$ is continuous with probability 1.
\end{lem}
The proof is standard,  see  Appendix.

Let $f$ be a c\`adl\`ag function, $f(0)=0.$ By $X_{f,x}(t)$
 denote a solution to the integral equation
 \[
X_{f,x}(t)= x+\int_0^t A(X_{f,x}(s)) ds + f(t).
\]
\begin{remk}
Function $X_{f,x}(t)$ is well-defined until it hits 0.
Note that if $f=0$ and $x\neq0,$ then the function $ X_{0, x}$ never hits 0 because  $A_{rad}(x)\sim \bar a(\frac{x}{|x|}) x^\beta$ as $x\to0$, where $\bar a$ is positive and continuous.
\end{remk}

\begin{lem}\label{lem:int-eq-cont}
Let Assumption \ref{assum:1} be satisfied. Then for any $x_0\neq 0$ and $T>0$ we have  the uniform convergence on $[0,T]$:
\[
\lim\limits_{\substack{\|f\|_{[0,T]}\to0\\ x \to x_0}}  \|X_{f,x}- X_{0,x_0}\|_{[0,T]}=0,
\]
where $\|g\|_{[0,T]}:=\sup_{[0,T]}|g(t)|.$

Moreover,
\[
\forall T>0\ \forall \mu>0
\ \forall \delta\in(0,1)   \ \exists \ve >0\ \forall f, \ \|f\|_{[0,T]}\leq \ve \
 \forall x,  |x|\in[\delta, \delta^{-1}]\ \forall y, \ |y-x|\leq \ve: 
 \]
\[
\sup_{t\in[0,T]}\left| X_{f,y}(t)-X_{0,x}(t)\right|\leq \mu.
\]
\end{lem}


The proof of this Lemma is standard. Indeed, observe  that 
Assumption \ref{assum:1} ensures that $X_{0,x}(t), t\in[0,T]$ is separated from  0. 
Since $A$ satisfies the local 
Lipschitz condition 
outside of 0 and linear growth condition, the proof follows from the application of Gronwall's lemma.

Let us denote by $X_{\ve, x}, \ve\geq 0$ the solution to  \eqref{eq:perturbed} with
 initial condition $X_{\ve, x}(0)=x\neq 0$ (I.e., $X_{\ve, x}(t)$ is $X_{\ve B_{\alpha}, x}(t)$ from Lemma \ref{lem:int-eq-cont}. We hope that the reader will not be confused by our notations).

\begin{corl}\label{corl:3.2}
Let the assumptions of Theorem \ref{thm:main_limit_small_noise} be satisfied. Then for any $x_0\neq 0$ we have a.s.
 convergence of stochastic processes:
\[
\lim\limits_{\substack{\ve\to0\\ x \to x_0}}  X_{\ve,x}(\cdot)  =  X_{0,x_0}(\cdot)\ \ \mbox{ in  } \    D([0,\infty)),
\]
where $D([0,\infty))=D([0,\infty),\mbR^d)$ denotes the Skorohod space of c\`adl\`ag functions.
Moreover,
\[
\forall T>0\ \forall \mu>0
\ \forall \delta \in(0,1)\ \exists \nu>0\   \exists \epsilon_0=\epsilon_0(\mu, \delta)>0\
 \forall \ve\in [0,\epsilon_0]\ 
\]
\bel{eq:3.5.0}
\forall x, |x|\in [\delta,\frac1\delta]\ \forall y,\ |y-x|\leq \nu:\ \ \ \ \Pb\left(  \sup_{t\in[0,T]}\left| X_{\ve,y}(t)-X_{0,x}(t)\right|\geq \mu \right)\leq \mu.
\ee 
In particular,
if random variables $\{\zeta_\ve, \ve\geq 0\}$ are independent of the noise and 
 $\zeta_\ve\Rightarrow \zeta_0$ as $ \ve\to0,$ where
$\zeta_0\neq0$ a.s., then
\bel{eq:720}
X_{\ve,\zeta_\ve} \Rightarrow   X_{0,\zeta_0} ,\ \mbox{ as } \ \ve\to0 \ \mbox{ in  } \ D([0,\infty)).
\ee
\end{corl}
\begin{remk}\label{remk:1182}
Equation \eqref{eq:3.5.0} may be convenient to write in terms of absolute value of vectors $x,y$ and their 
polar angles (with possibly another $\nu$):
\be\label{eq:3.5.0_ext}
\begin{aligned}
\forall x, |x|\in [\delta,\frac1\delta]\ \forall y\
\mbox{ such that } \ \left|\frac{|x|}{|y|}-1\right|\leq \nu \mbox{ and } 
\left|\frac{x}{|x|}-\frac{y}{|y|}\right|\leq \nu:\\
\ \ \ \ \ \Pb\left(  \sup_{t\in[0,T]}\left| X_{\ve,y}(t)-X_{0,x}(t)\right|\geq \mu \right)\leq \mu.
\end{aligned}
\ee
 \end{remk}



\vskip15pt
 Consider the following time-space transformation. Set
\bel{eq:space-time}
 \wt X_\ve(t):= \ve^{\frac{-\alpha}{  \alpha+\beta-1}}   X_\ve (\ve^{\frac{\alpha(1-\beta)}{ \alpha+\beta-1}}t).
 \ee
 Then
 \[
 \wt  X_\ve(t) = \ve^{\frac{-\alpha}{  \alpha+\beta-1}}   X_\ve(\ve^{\frac{\alpha(1-\beta)}{ \alpha+\beta-1}}t)=
 \]
 \[
 \ve^{\frac{-\alpha}{  \alpha+\beta-1}}   \int_0^{\ve^{\frac{\alpha(1-\beta)}{ \alpha+\beta-1}}t}
 A(X_\ve (s))ds + \ve^{1+\frac{-\alpha}{  \alpha+\beta-1}} B_\alpha(\ve^{\frac{\alpha(1-\beta)}{ \alpha+\beta-1}}t)=
 \]
 \[
 \ve^{\frac{-\alpha}{  \alpha+\beta-1}} \ve^{\frac{\alpha(1-\beta)}{ \alpha+\beta-1}}  \int_0^{t}
 A(X_\ve (\ve^{\frac{\alpha(1-\beta)}{ \alpha+\beta-1}}z))dz + \ve^{\frac{\beta-1}{  \alpha+\beta-1}} B_\alpha(\ve^{\frac{\alpha(1-\beta)}{ \alpha+\beta-1}}t)=
 \]
 \[
 \ve^{\frac{ -\alpha\beta}{  \alpha+\beta-1}}   \int_0^{t}
 A(\ve^{\frac{ \alpha}{  \alpha+\beta-1}} \wt  X_\ve (z))dz +  B^{(\ve)}_\alpha(t),
 \]
where $\{B^{(\ve)}_\alpha(t)\}:=\{\ve^{\frac{\beta-1}{  \alpha+\beta-1}} B_\alpha(\ve^{\frac{\alpha(1-\beta)}{ \alpha+\beta-1}}t)\} \overset{d}= \{B_\alpha(t)\}$ is a $d$-dimensional symmetric L\'evy $\alpha$-stable process.

We have, see   \eqref{eq:A_representation},
\bel{eq:657}
d\wt  X_\ve(t) =  \ve^{\frac{ -\alpha\beta}{  \alpha+\beta-1}}
 \left(  A_{rad}(\ve^{\frac{ \alpha}{  \alpha+\beta-1}} \wt  X_\ve (t)) +
A_{tan} (\ve^{\frac{ \alpha}{  \alpha+\beta-1}} \wt  X_\ve (t))\right)dt+ dB^{(\ve)}_\alpha(t).
\ee

Further we will skip $\ve$ in     notation $\{B^{(\ve)}_\alpha(t)\}$.

 It follows from Assumption \ref{assum:1}  that for any $x\neq 0:$
 \bel{eq:A_rad_ve}
   \ve^{\frac{ -\alpha\beta}{  \alpha+\beta-1}}    A_{rad}(\ve^{\frac{ \alpha}{  \alpha+\beta-1}} x)
 =a(\ve^{\frac{ \alpha}{  \alpha+\beta-1}} x) x^\beta \to \bar a(\frac{x}{|x|})x^\beta
=\bar a(\vf)x^\beta,\ \ve\to0;
 \ee
 \bel{eq:A_tan_ve}
  \ve^{\frac{ -\alpha\beta}{  \alpha+\beta-1}}    A_{tan}(\ve^{\frac{ \alpha}{  \alpha+\beta-1}} x) =
 {\ve^{\frac{ \alpha \gamma}{  \alpha+\beta-1}} \frac{A_{tan} (\ve^{\frac{ \alpha}{  \alpha+\beta-1}} x)}
{\left(\ve^{\frac{ \alpha}{  \alpha+\beta-1}} |x|\right)^{\beta+\gamma}} |x|^{\beta+\gamma}}\to0,\ \ve\to0
.
\ee
\begin{lem}\label{lem:3.1}
Let $\alpha+\beta>1, \alpha\in(1,2], |\beta|<1, $ and $Y_\ve(t), t\geq 0, \ve \geq 0$ be solutions to  the SDEs:
\[
dY_\ve(t)= b_\ve(Y_\ve(t))dt + d B_\alpha(t),\ t\geq 0,\ \ \  \ Y_\ve(0)=0,
\]
where $b_\ve$ are continuous functions on $\mbR^d\setminus\{0\}$.

Assume that 

1) $\lim_{\ve\to0}b_\ve(x)=b_0(x),$ uniformly on every compact that does not contain 0,

2) $\exists C>0\ \forall \ve\geq 0\ \forall x\neq 0:\ \ |b_\ve(x)|\leq C(1+|x|^\beta).$

Then 
\bel{eq:1295}
Y_\ve\Rightarrow Y_0 \ \mbox{ as } \  \ve\to0 \ \mbox{ in } \  D([0,\infty)).
\ee

In particular, we have the weak convergence of stochastic processes:
\[
\wt{X}_\ve \Rightarrow  \bar X \ \mbox{ as } \ \ve\to0 \mbox{ in  } \    D([0,\infty)),
\]
where $\bar X$ is defined in
\eqref{eq:SDE_lim0}, \eqref{eq:SDE_lim0init}.
\end{lem}
\begin{proof}[Proof of Lemma \ref{lem:3.1}] 
Using localization technique and sub-linear growth of coefficients,  without loss of generality we may assume
 that
supports of all $\{b_\ve\}$ are included in the same compact.

 It follows from  \cite{Portenko1995}  that $\{Y_\ve(t)\}$ generates Feller semigroup 
and for any bounded uniformly continuous function 
 $f$ we have the uniform convergence:
 \[
 \sup_x|\E_x f(Y_\ve(t))- \E_x f(Y_0(t))|\to 0, \ \ve\to0,
 \]
 see   \cite[Lemma 2]{Portenko1995}.  
The application of \cite[Theorem 2.5, Chapter 4]{EK} yields \eqref{eq:1295}. 
\end{proof}
\begin{remk}
    Only the case $\alpha\in(1,2)$ was considered in \cite{Portenko1995}. The case $\alpha=2,$ i.e. the case when the noise is a Brownian motion, follows from the same reasoning and estimates for transition probability density functions.
\end{remk}

Let us  sketch a  proof of  Theorem \ref{thm:main_limit_small_noise} again in more details.

Recall that $\wt  X_\ve(t):= \ve^{\frac{-\alpha}{  \alpha+\beta-1}}
  X_\ve (\ve^{\frac{\alpha(1-\beta)}{ \alpha+\beta-1}}t).$

By  $\wt X_{\ve,x} $ denote a solution to \eqref{eq:657} with initial condition
 $\wt X_{\ve,x}(0)=x$ (so $\wt X_{\ve }=\wt X_{\ve,0}$.)

For $M\geq 0$  set
\[
\tau_M^\ve:= \inf\{t\geq0\ : \ |X_\ve(t)|\geq M\},
\]
\[
\tau_M^{\ve,x}:= \inf\{t\geq0\ : \ |X_{\ve,x}(t)|\geq M\},
\]
\[
\wt \tau_M^\ve:= \inf\{t\geq0\ : \ |\wt  X_\ve(t)|\geq M\},
\]
\[
\wt \tau_M^{\ve,x}:= \inf\{t\geq0\ : \ |\wt X_{\ve,x}(t)|\geq M\}.
\]

Note that
\bel{eq:3.6.-1}
\wt \tau_{M}^\ve = \ve^{\frac{-\alpha(1-\beta)}{\alpha+\beta-1} }
 \tau_{\ve^{\frac{\alpha}{\alpha+\beta-1}}M}^\ve,\ \ \      
     \wt  X_\ve(\wt \tau_{M}^\ve)= \ve^{\frac{-\alpha}{\alpha+\beta-1}} 
  X_\ve( \tau_{\ve^{\frac{\alpha}{\alpha+\beta-1}}M}^\ve).
\ee

{\it Step 1.} Let $R>0$ be a large number,  $\ve >0$ be a small number. 
We apply Lemma \ref{lem:3.1} and Theorem \ref{thm:integral_eq_growth} and show that
if $R$ is large and $\ve>0$ is small, then
\[
|\wt  X_\ve(\wt \tau_R^\ve)| \approx R \ \ \mbox{and} \ \
\frac{\wt  X_\ve(\wt \tau_R^\ve)}{|\wt  X_\ve(\wt \tau_R^\ve)|}
 \overset{d}\approx \bar\Phi(+\infty),
\]
where $\bar\Phi(+\infty)$ is defined in \eqref{eq:25.1}.

So
\bel{eq:3.6.0}
| X_\ve(\tau_{\ve^{\frac{\alpha}{\alpha+\beta-1}}R}^\ve)| \approx \ve^{\frac{ \alpha}{\alpha+\beta-1}} R
\ee
and
\bel{eq:3.6.1}
\frac{ X_\ve( \tau^\ve_{\ve^{\frac{\alpha}{\alpha+\beta-1}}R}) }{|{ X_\ve( \tau^\ve_{\ve^{\frac{\alpha}{\alpha+\beta-1}}R}) }|}
=
\frac{\wt  X_\ve( \wt \tau^\ve_{R}) }{|{\wt  X_\ve(\wt \tau^\ve_{R})}|}
\overset{d}\approx \bar \Phi(+\infty).
\ee

{\it Step 2.} 
  Let $R>0$ be a large number, $\delta>0$ be a small number, $\ve \ll \delta$.
We show that if $|x|\in[ \ve^{\frac{\alpha}{\alpha+\beta-1}}R, \delta]$,
then
\bel{eq:3.6.2}
|X_{\ve,x}(\tau_\delta^{\ve,x})|\approx \delta\ \ \mbox{and } \ \ \frac{X_{\ve,x}(\tau_\delta^{\ve,x})}{|X_{\ve,x}(\tau_\delta^{\ve,x})|}\approx \frac{x}{|x|}.
\ee
The first approximate equality in \eqref{eq:3.6.1} means that $X$ cannot exit from the ball $\{y: \ \|y
\|\leq \delta \}$
by a large jump. The second equality means that the process $X_{\ve,x}$   rotates  slightly in the set 
$\{y: \ \ve^{\frac{\alpha}{\alpha+\beta-1}}R\leq \|y\|\leq \delta\}$.  This happens because the effect of drift
there dominates the effect of the noise if $R$ is large enough. It will be seen from the proof that the effect of the drift and the noise comparable in the set 
$\{y: \ \|y\|<\ve^{\frac{\alpha}{\alpha+\beta-1}}R \}$.

It follows from \eqref{eq:3.6.1}, \eqref{eq:3.6.2}, and Corollary \ref{corl:3.2} that
\bel{eq:3.6.3}
|X_{\ve}(\tau_\delta)| \approx \delta,\ \ \ \          \frac{X_{\ve}(\tau_\delta)}{|{X_{\ve}(\tau_\delta)}|}\overset{d}\approx \bar\Phi(+\infty).
\ee

{\it Step 3.} We shall deduce from  Corollary  \ref{corl:3.2}  (see \eqref{eq:720}) and \eqref{eq:3.6.3}
that
\bel{eq:3.7.0}
X_\ve(\tau_\delta^\ve+t) \overset{d}\approx  X_{0,\delta \bar\Phi(+\infty)}(t).
\ee

{\it Step 4.} We prove that $\tau_\delta^\ve$ is small if $\ve$ and $\delta$ are small. So
\bel{eq:3.7.1}
X_\ve(\tau_\delta^\ve+t)
\approx X_\ve(t).
\ee

{\it Step 5.}
It follows from Assumption \ref{assum:2}   that
\bel{eq:3.7.2}
X_{0,\delta\bar \Phi (+\infty)}(t)
\approx X_0(t, \bar\Phi(+\infty)) 
\ee
if $\delta$ is small.

So, \eqref{eq:3.7.0}, \eqref{eq:3.7.1}, and \eqref{eq:3.7.2} yield the approximate equality
\[
X_{\ve}(t)
\overset{d}\approx X_0(t, \bar\Phi(+\infty)).
\]

\subsection{End of the proof of Theorem \ref{thm:main_limit_small_noise}}\label{subsect:proof of Th4_1}

Let $d(\xi, \eta)=d(P_\xi, P_\eta)$ be the Levy-Prokhorov metric between distributions
of random variables or random vectors $\xi$ and $\eta$.

Denote by $\mathcal{P}(P_\xi,P_\eta)$ the set of all pairs $(\tilde \xi, \tilde \eta)$ (defined maybe on different probability spaces) with marginals $P_\xi$ and $P_\eta$.
It follows from Corollary 11.6.4 \cite{Dudley} that
\bel{eq:LevyProkhDud}
d(\xi, \eta)=\inf \{ \mu>0\ :\ \exists (\tilde \xi, \tilde \eta)\in \mathcal{P}(P_\xi,P_\eta) \mbox{ such that } \Pb(|\tilde \xi- \tilde \eta|> \mu)\leq \mu\}.
\ee
\begin{lem}\label{lem:3.3} For any $  \mu>0 $ there exists $ R_1>0$ such that for all $ R\geq R_1:$
\bel{eq:1070}
\ba
d\left(\frac{\bar X(\bar \tau_R )}{|{\bar X(\bar \tau_R )}|} , \bar\Phi(+\infty)\right)<\mu;
\\
d\left(\frac{\bar X(\bar \tau_R )}{R} ,1 \right)<\mu;
\\
d\left(\frac{R^{1-\beta}}{\bar \tau_R } , (1-\beta) \bar a(\bar \Phi(+\infty))\right)<\mu;\ \
\ea
\ee
and there is $  \epsilon_1=\epsilon_1(R)>0$ such that for all $ \ve\in(0,\epsilon_1]:$
\bel{eq:1080}
\ba
d\left(\frac{\wt  X_\ve(\wt \tau_R^\ve)}{|{\wt  X_\ve(\wt \tau_R^\ve)}|} , \bar\Phi(+\infty)\right)<\mu;
\\
 d\left(\frac{{|\wt  X_\ve(\wt \tau_R^\ve)}|}{R} ,1 \right)<\mu;
\\
 d\left(\frac{R^{1-\beta}}{\wt \tau_R^\ve} , (1-\beta) \bar a(\bar \Phi(+\infty))\right)<\mu.
\ea
\ee
\end{lem}
\begin{proof}
It follows from Theorems \ref{thm:integral_eq_growth} and \ref{thm:SDE_growth} (see Example \ref{expl:stable})
that 
\newline
$\frac{\bar X(t )}{|{\bar X(t)}|} \to \bar\Phi(+\infty), \ t\to\infty$ a.s. and
$|\bar X(t )| \sim ((1-\beta)  \bar a(\bar \Phi(+\infty)) t)^{\frac{1}{1-\beta}}, \ t\to\infty$  a.s.

This yields \eqref{eq:1070}.
Equations \eqref{eq:1080}  follows from \eqref{eq:1070} and Lemma \ref{lem:3.1}.
\end{proof}
The next lemma is the most important part of the proof. Informally, the statement means  that 
if $\ve$, $\delta$ are small, $R$ is large  and $|x|\geq \ve^{\frac{\alpha}{\alpha+\beta-1}}R$ (see Step  2 above), then 
\eqref{eq:3.6.2} holds true, i.e.,  the process can't exit from the ball 
$\{y\ : |y|\leq \delta\}$ by a large jump and the radius at the instant of exit  approximately
equals $\delta;$ moreover, the polar angle at the instant of exit approximately equals initial polar angle $\frac{x}{|x|}$. 
We also will prove  that the exit time $\tau_\delta^\ve$ is small if $\ve$ and $\delta$ are small (see Step  4 above).
\begin{lem}\label{lem:3.4}
\begin{multline*}
    \forall \mu>0\   \exists \delta_2=\delta_2(\mu)>0\ \exists R_2=R_2(\mu)>0\
 \exists \epsilon_2=\epsilon_2(\mu)>0
\\
\forall R>R_2, \ \delta\in(0,\delta_2),\  \ve\in (0,\epsilon_2) \
 \mbox{such that }\ R \ve^{\frac{\alpha}{\alpha+\beta-1}}<\delta, 
 \\
 \ \forall x\in\mbR^d\
 \mbox{such that } \ R \ve^{\frac{\alpha}{\alpha+\beta-1}}<|x|<\delta :
\end{multline*}
\bel{eq:lem3.5.1}
\Pb\left(\left|
\frac{X_{\ve,x}(\tau_\delta^{\ve,x})}{|X_{\ve,x}(\tau_\delta^{\ve,x})|} -\frac{x}{|x|}\right|>\mu\right)<\mu;
\ee
\bel{eq:lem3.5.2}
\Pb\left(\left|
\frac{|X_{\ve,x}(\tau_\delta^{\ve,x})|}{\delta} -1\right|>\mu\right)<\mu;
\ee
\bel{eq:lem3.5.3}
\Pb\left(\left|
\tau_\delta^{\ve,x}\right|>\mu\right)<\mu.
\ee
\end{lem}
As for the proof see the Appendix.

Let $R>0$ be a fixed (large) number. It follows from equality \eqref{eq:3.6.-1} 
and the strong Markov property
\begin{multline*}
    d\left(
\frac{X_{\ve}(\tau_\delta^\ve)}{|X_{\ve}(\tau_\delta^\ve)|}, \bar\Phi(+\infty) \right)\leq \\
d\left(\frac{X_{\ve}(\tau_\delta^\ve)}{|X_{\ve}(\tau_\delta^\ve)|}, \frac{X_\ve( \tau_{\ve^{\frac{\alpha}{\alpha+\beta-1}}R}^\ve)}{|X_\ve( \tau_{\ve^{\frac{\alpha}{\alpha+\beta-1}}R}^\ve)|} \right)
+d\left(\frac{X_\ve( \tau_{\ve^{\frac{\alpha}{\alpha+\beta-1}}R}^\ve)}{|X_\ve( \tau_{\ve^{\frac{\alpha}{\alpha+\beta-1}}R}^\ve)|}, \bar\Phi(+\infty) \right)
=\\
d\left(\frac{X_{\ve}(\tau_\delta^\ve)}{|X_{\ve}(\tau_\delta^\ve)|}, \frac{X_\ve( \tau_{\ve^{\frac{\alpha}{\alpha+\beta-1}}R}^\ve)}{|X_\ve( \tau_{\ve^{\frac{\alpha}{\alpha+\beta-1}}R}^\ve)|} \right)
+d\left(\frac{\wt  X_\ve(\wt \tau_{R}^\ve)}{|\wt  X_\ve(\wt \tau_{R}^\ve)|}, \bar\Phi(+\infty) \right)=
\\
d\left(\frac{\hat X_{\ve,x}(\tau_\delta^{\ve,x})}{|\hat X_{\ve,x}(\tau_\delta^{\ve,x})|}\Big|_{x=X_\ve( \tau_{\ve^{\frac{\alpha}{\alpha+\beta-1}}R}^\ve)} , \frac{X_\ve( \tau_{\ve^{\frac{\alpha}{\alpha+\beta-1}}R}^\ve)}{|X_\ve( \tau_{\ve^{\frac{\alpha}{\alpha+\beta-1}}R}^\ve)|} \right)
+d\left(\frac{\wt  X_\ve(\wt \tau_{R}^\ve)}{|\wt  X_\ve(\wt \tau_{R}^\ve)|}, \bar\Phi(+\infty) \right),
  \end{multline*}
where $\hat X_{\ve,x}\overset{d}=  X_{\ve,x},$ $\hat X_{\ve,x}$ is independent of $X_\ve( \tau_{\ve^{\frac{\alpha}{\alpha+\beta-1}}R}^\ve).$

 Lemmas \ref{lem:3.3} and \ref{lem:3.4}
imply
\[
\forall \mu>0\ \exists \delta_3=\delta_3(\mu)>0 \ 
\forall \delta\in(0,\delta_3] \ \exists \epsilon_3=\epsilon_3(\mu,\delta)>0\ \forall \ve\in (0, \epsilon_3]:
\]
\bel{eq:3.9.1}
d\left(
\frac{X_{\ve}(\tau_\delta^\ve)}{|X_{\ve}(\tau_\delta^\ve)|}, \bar\Phi(+\infty) \right)<\mu,
\ee
and similarly,
\bel{eq:3.9.2}
\Pb\left(
\frac{|X_{\ve}(\tau_\delta^\ve)|}{\delta}> 1+\mu\right)<\mu;
\ee
\bel{eq:3.9.3}
\Pb\left( \tau_\delta^\ve>\mu\right)<\mu.
\ee

To prove  Theorem \ref{thm:main_limit_small_noise} it is sufficient to verify that any sequence $\{X_{\ve_n}\}$ contains a subsequence $\{X_{\ve_{n_k}}\}$ such that
\[
X_{\ve_{n_k}}(\cdot)\Rightarrow  X_0(\cdot, \bar \Phi(+\infty)),\ k\to\infty.
\]
It follows 
from Lemma \ref{lem:3.5} that $\{X_\ve\} $ is relatively compact. So  without loss of generality we will assume that $\{X_{\ve_n}\}$ is already convergent.


Let $\mu>0$ be fixed. Select $\delta_3=\delta_3(\mu)$
from \eqref{eq:3.9.1} -- \eqref{eq:3.9.3}, and for fixed $\delta\in (0,\delta_3)$ select
$ \epsilon_3=\epsilon_3(\mu,\delta)$.
Let $n_3=n_3(\mu)=n_3(\mu,\delta)$ be such that  $\ve_n <\epsilon_3(\mu)$ for any $n\geq n_3$.

It follows from \eqref{eq:LevyProkhDud} 
that for any $n\geq n_3$ there are   copies
$\hat \Phi(+\infty) \overset{d}= \bar \Phi(+\infty)$ and
$\hat \zeta_n \overset{d}=
 \frac{  X_{\ve_n}(  \tau_\delta^{\ve_n})}{|  X_{\ve_n}(  \tau_\delta^{\ve_n})|}$
 defined on some probability space such that for all $n\geq n_3$
\bel{eq:1655}
\Pb\left(\left|
\hat \zeta_n -    \hat \Phi(+\infty)     \right|> 2\mu \right) < 2\mu.
\ee
\begin{lem}\label{lem:coupling}
    Let $(\xi_n,\eta_n), n\geq 1,$ be a sequence of  random elements defined in possibly  different probability spaces. Assume that
    
    1) all $\xi_n$ take values in a complete separable metric space $G$ and have same distribution,

    2) $\eta_n$ takes values in a complete separable metric space $G_n, n\geq 1.$

    Then there are random elements   $\wt \xi, \wt \eta_n, n\geq 1,$ defined on the same probability space such that $(\xi_n,\eta_n) \overset{d}{=} (\wt \xi, \wt \eta_n), n\geq 1.$
\end{lem}
\begin{proof}
   We will give a proof for completeness, although assertions of this type are well known in coupling theory, cf. \cite{Thorisson}. Denote by
   $P_\xi(dx)$ the distribution of $\xi$ in $E$ and 
by $P_n(x,dy_n)$ the regular conditional distribution of $\eta_n$ given $\xi_n=x.$ Since all metric spaces are assumed to be complete and separable, the regular conditional distributions exist.

Consider the product-space  $E\times G_1\times G_2\times\dots$ with   Borel $\sigma-$algebra and define the probability measure as
 \[
 P(dx, d y_1,dy_2,\dots)=P_\xi(dx)P_1(x,dy_1)P_n(x,dy_2)\dots.
 \]
 Then the coordinate random elements $\wt \xi:=x, \wt \eta_n:=y_n, n\geq 1,$ satisfy conditions of the lemma.
\end{proof}
 
Using Lemma \ref{lem:coupling} 
we can construct processes  copies $\hat X_{\ve_n}$ on a joint probability space such that
$\hat \zeta_n =  \frac{\hat X_{\ve_n}(\hat \tau_\delta^{\ve_n})}{|\hat X_{\ve_n}(\hat \tau_\delta^{\ve_n})|} $
a.s. and \eqref{eq:1655} holds true.

Therefore, \eqref{eq:3.9.1} -- \eqref{eq:3.9.3} yield
\bel{eq:3.10.1}
\Pb\left(\left|  \frac{\hat X_{\ve_n}(\hat \tau_\delta^{\ve_n})}{|\hat X_{\ve_n}(\hat \tau_\delta^{\ve_n})|} -    \hat \Phi(+\infty)     \right|> 2\mu \right) < 2\mu,
\ee
\bel{eq:3.10.2}
\Pb\left(  \frac{|\hat X_{\ve_n}(\hat \tau_\delta^{\ve_n})|}{\delta}> 1+\mu \right) < \mu,
\ee
\bel{eq:3.10.3}
\Pb\left(\hat \tau_\delta^{\ve_n}> \mu \right) < \mu.
\ee

We have
\bel{eq:3.11.0}
\sup_{t\in[0,T]}\left| \hat X_{\ve_n}(t) - X_0(t,\hat \Phi (+\infty))\right|\leq
\sup_{t\in[0, \hat \tau_\delta^{\ve_n}]}\left| \hat X_{\ve_n}(t)\right|
+\sup_{t\in[0, \hat \tau_\delta^{\ve_n}]}\left| X_0(t,\hat \Phi (+\infty))\right|+
\ee
\[
\sup_{t\in[0,T]}\left| \hat X_{\ve_n}(\hat \tau_\delta^{\ve_n}+t) - X_0(\hat \tau_\delta^{\ve_n}+t,\hat \Phi (+\infty))\right|=
\]
\[
\left| \hat X_{\ve_n}( \hat \tau_\delta^{\ve_n})\right|
+\left| X_0( \hat \tau_\delta^{\ve_n},\hat \Phi (+\infty))\right|+
\sup_{t\in[0,T]}\left| \hat X_{\ve_n}(\hat \tau_\delta^{\ve_n}+t) - X_0(\hat \tau_\delta^{\ve_n}+t,\hat \Phi (+\infty))\right|=: I_1^{n,\delta} + I_2^{n,\delta}+ I_3^{n,\delta}.
\]
It follows from \eqref{eq:3.9.2} that
\bel{eq:3.11.1}
\forall n\geq n_3\ 
:\ \ \ \Pb\left( I_1^{n,\delta}> \delta(1+\mu)\right)< \mu.
\ee
It is easy to see that there exists $K>0$ such that $|X_0(t,\vf)|\leq K t^{\frac{1}{1-\beta}}$
 for small $t\geq 0$, where $K$ is independent of  $\vf.$ So,
\bel{eq:3.11.1.1}
\Pb( I_2^{n,\delta}\geq K\mu^{\frac{1}{1-\beta}}) =
\Pb\left(\left| X_0( \hat \tau_\delta^{\ve_n},\hat \Phi (+\infty))\right| > K\mu^{\frac{1}{1-\beta}}\right) \leq
\Pb\left( \hat \tau_\delta^{\ve_n} > \mu\right)<\mu.
\ee
The last inequality follows from \eqref{eq:3.10.3}.

Consider $I_3^{n,\delta}$. By the strong Markov property we have

{
\[
I_3^{n,\delta}\overset{d}=
\sup_{t\in[0,T]}\left|X_{\ve_n,y}(t)|_{y=\hat X_{\ve_n}(\hat \tau_\delta^{\ve_n})} -
 X_{0,x}(t)|_{x=X_{0}(\hat \tau_\delta^{\ve_n},\hat \Phi (+\infty))}\right|,
\]
where $\{X_{\ve_n,y}(t)\}$ is independent of $\hat X_{\ve_n}(\hat \tau_\delta^{\ve_n})$.
\newline
It follows from Corollary \ref{corl:3.2}, Remark \ref{remk:1182} and Lemma \ref{lem:3.4} that 
\bel{eq:1650}
\Pb(I_3^{n,\delta}> \mu)< \mu
\ee
 if $\delta$ is sufficiently small and $n$ is sufficiently large.
Let $\mu>0$ be fixed. Combining estimates \eqref{eq:3.11.0}, \eqref{eq:3.11.1}, \eqref{eq:1650} we 
}
can select sufficiently small $\delta=\delta(\mu)>0$  and  $N=N(\delta,\mu)$ such that 
 \[
\forall n\geq N:  \Pb\left( \sup_{t\in[0,T]}\left| \hat X_{\ve_n}(t) - X_0(t,\hat \Phi(+\infty))\right|> 100\mu + K\mu^{\frac{1}{1-\beta}}\right)< 100\mu.
\]
The last display yields the uniform convergence in probability
\[
\sup_{t\in[0,T]}\left| \hat X_{\ve_n}(t) - X_0(t,\hat \Phi(+\infty))\right|\overset{\Pb}\to0, n\to\infty.
\]
Since $\hat X_{\ve_n}\overset{d}= X_{\ve_n}, \ X_0(t,\hat \Phi(+\infty))\overset{d}=X_0(t,  \Phi(+\infty))$
this concludes the proof of 
 Theorem \ref{thm:main_limit_small_noise}.

\section{Appendix}

\begin{proof}[  Proof of Theorem \ref{thm:integral_eq_growth}]
Let $\{Z(t)\}$ be a solution to \eqref{eq:integral_eq}. Set $\tilde Z(t):=Z(t)-\xi(t).$
Then
\bel{eq:4.1.0}
 d\tilde Z(t) = A(\tilde Z(t) +\xi(t))dt=A(Z(t))dt,\ t\geq 0.
\ee
Assume for a while that $\tilde Z(t)\neq 0, t\geq 0.$
Then for $\tilde r(t)=|\tilde Z(t)|,\ r(t)=|Z(t)|  $ we have:
\[
d\tilde r^{1-\beta}(t)= (1-\beta) \tilde r^{-\beta}(t)
\frac{\lg \tilde Z(t),  A( Z(t))\rg}
{|\tilde Z(t)|} dt = (1-\beta) \tilde r^{-\beta}(t)
\frac{\lg \tilde Z(t),  A_{rad}( Z(t))+A_{tan}( Z(t))\rg}
{|\tilde Z(t)|} dt=
\]
\bel{eq:4.1.2}
(1-\beta) \tilde r^{-\beta-1}(t)  
{\lg \tilde Z(t),   a(Z(t))\frac{Z(t)}{|{Z(t)}|}\rg} 
 dt+  
(1-\beta) \tilde r^{-\beta-1}(t)
{\lg \tilde Z(t),  A_{tan}( Z(t))\rg}
 dt.
\ee
%
It follows from \eqref{eq:298} that 
\[
\exists R_{rad}>0\  \exists C_{rad}>0 \ \forall x, \ |x|\geq R_{rad}\ :\ \
a(x)\geq C_{rad} |x|^\beta.
\]
 Suppose that for some $T>0$ we have
\bel{eq:assumption}
|\xi(t)|\leq \frac{|\tilde Z(t)|}{2}  \ \mbox{ and }
 \ |\tilde Z(t)| \geq 2(R_{rad}\vee R_{tan}) \mbox{ for } t\in [0,T],
\ee
where $R_{tan}$ is from \eqref{eq:A_t_infty}.
 Then 
\[
|Z(t)|/2\leq |\tilde Z(t)|\leq 2 |Z(t)|, \ \ \lg Z(t), \tilde Z(t) \rg\geq 
|\tilde Z(t)|/2, \ , \ t\in [0,T],
\] 
 and assumptions of Theorem \ref{thm:integral_eq_growth} yield the inequality
\bel{eq:4.2.17}
d\tilde r^{1-\beta}(t)\geq 
C_1 \tilde r^{-\beta-1}(t)
|\tilde Z(t)| C_{rad} |\tilde Z(t)|^\beta
 dt  - C_2 
 \tilde r^{-\beta-1}(t)
 {|\tilde Z(t)| C_{tan} |\tilde Z(t)|^{\beta-\gamma}}
 dt=
\ee
\[
C_1   C_{rad}  
 dt  - C_2 C_{tan} 
\tilde r^{- \gamma} dt,\   \ t\in [0,T],
\]
 where $C_1, C_2$ are some constants.
 
Thus, there are constants  $R_1=R_1(A)\geq R_{rad}\vee R_{tan},\  K=K(A, R_1)>0$, which are independent of $T$, and such that
\bel{eq:4.2.1}
  \forall r(0)\geq R_1\ :
 \ \  \tilde r^{1-\beta}(t)\geq \tilde r^{1-\beta}(0)+
Kt,\  t\in [0,T].
\ee
\begin{lem}\label{lem:4.1}
Suppose that assumptions of Theorem  \ref{thm:integral_eq_growth}
are satisfied. Then there is $R_0$ such that \eqref{eq:assumption}
and \eqref{eq:4.2.1} are true for any $T> 0 $ whenever $|\tilde Z(0)|> R_0$.

\end{lem}
\begin{remk}
In this lemma we assume only existence of a solution but not a uniqueness.
\end{remk}
\begin{proof}
Let $|\tilde Z(0)|>  2 R_{rad}\vee 2R_{tan}.$ It follows from 
\eqref{eq:4.2.17} that \eqref{eq:assumption} and \eqref{eq:4.2.1} may fail only  if 
there is $t$ such that $|\xi(t)|> \frac{|\tilde Z(t)|}{2}.$

 It follows from \eqref{eq:estim_noise}  that
\bel{eq:1084}
\exists R_2 \ \forall t\geq 0: \ \ 2|\xi(t)|\leq (R_2^{1-\beta}+  {Kt})^{\frac{1}{1-\beta}}-1,
\ee
where $K$ is from \eqref{eq:4.2.1}.

Assume that $|\tilde Z(0)|> R_0:= R_2\vee 2 R_{rad}\vee 2R_{tan}.$  Set
\bel{eq:4.2.5}
t_0:=\inf\left\{ t \geq 0\ :\ |\xi(t)| > \frac{|\tilde Z(t)|}{2}\right\}=
\inf\left\{ t \geq 0\ :\ 2|\xi(t)| >  |\tilde r(t)| \right\}.
\ee
Note that $t_0\neq 0$ because $\tilde Z$  is continuous.

To prove the Lemma  it suffices to verify that $t_0=\infty.$
Assume the converse, i.e., $t_0\in (0,\infty).$

It follows from \eqref{eq:4.2.1} that 
 $\tilde r(t)=|\tilde Z(t)| \geq 2(R_{rad}\vee R_{tan}), t\in [0,t_0]$ and
  for any $t\in [0, t_0]:$
\[
\tilde r(t)\geq (\tilde r^{1-\beta}(0)+Kt)^{\frac{1}{1-\beta}} \geq
(R_0^{1-\beta}+Kt)^{\frac{1}{1-\beta}}.
\]
Since $\tilde r$ is continuous,
\[
\exists \ve>0\ \forall t\in[t_0, t_0+\ve]\ :\ \ \tilde r(t)\geq
 (R_0^{1-\beta}+Kt)^{\frac{1}{1-\beta}}-1.
\]
This and \eqref{eq:1084} imply
\[
\tilde r(t)\geq 2|\xi(t)|,\ t\in[t_0,t_0+\ve],
\]
and we get a contradiction with the definition of $t_0.$

This proves the Lemma.
\end{proof}
Let's continue the proof of the theorem. Without loss of generality we will  further assume that $\xi(0)=0.$ This yields $ Z(0)=\tilde Z(0)$
and simplifies some constants.

Recall that
\[
|Z(t)|= |\tilde Z(t)+\xi(t)|\geq |\tilde Z(t)| - |\xi(t)|.
\]
If $ |Z(0)| >R_0,$ where $R_0$ is from  Lemma  \ref{lem:4.1}, then Lemma  \ref{lem:4.1} and
 the last inequality yield
\[
|Z(t)|\geq |\tilde Z(t)| /2\geq (   |Z(0)|^{1-\beta}+
Kt)^{\frac{1}{1-\beta}}/2, \ t\geq 0
\]
and the first part of the Theorem is proved.

\begin{lem}\label{lem:4.2}
Let   $Z(0)=x$, $|x|> R_0$, where $R_0$ is from Lemma \ref{lem:4.1}.
Then  there exists a limit
\[
\vf_{Z_x}(+\infty):=\lim_{t\to\infty}\frac{Z_x(t)}{|Z_x(t)|}.
\]
\end{lem}
\begin{proof}
It follows from Lemma \ref{lem:4.1} that $t^{\frac{1}{1-\beta}}=O(|Z(t)|)= O(|\tilde Z(t)|), \ t\to\infty.$
So
\bel{eq:4.3.1}
|Z(t)-\tilde Z(t)| =|\xi(t)|=o (|Z(t)|)= o(|\tilde Z(t)|),\ t\to\infty.
\ee
Hence, it is sufficient to verify existence of the limit
$\tilde \vf(t) : =\frac{\tilde Z(t)}{|\tilde Z(t)|}
$ as $t\to\infty.$ This limit will automatically coincide with   $\lim_{t\to\infty}\frac{ Z(t)}{|Z(t)|}.$

It follows from the proof of Lemma \ref{lem:4.1} that $\tilde Z(t)\neq 0$
 for all $t\geq 0$.

Since $\tilde \vf(t)$ is absolutely continuous function, in order
to prove existence of $\lim_{t\to\infty }\tilde \vf(t)$ it suffices to show that
\bel{eq:4.4.1}
\int_1^\infty |\frac{d \tilde \vf(t)}{dt}| dt<\infty.
\ee
We have
\bel{eq:4.4.2}
\frac{d \tilde \vf(t)}{dt} =\frac{d }{dt} \left( \frac{ \tilde Z(t)}{|\tilde Z(t)|}\right)=\frac{|\tilde Z(t)|^2 I_d- \tilde Z(t) \tilde Z^T(t)}{|\tilde Z(t)|^3}\cdot \frac{d \tilde Z(t)}{dt},
\ee
where $I_d$ is $d\times d$ identity matrix. 

So, \eqref{eq:4.1.0}, \eqref{eq:4.3.1},   \eqref{eq:4.4.2}, and Lemma \ref{lem:4.1} yield
that for $t\geq 1:$
\bel{eq:1356}
\left|\frac{d \tilde \vf(t)}{dt}\right| =\left| \frac{|\tilde Z(t)|^2 I_d- \tilde Z(t) \tilde Z^T(t)}{|\tilde Z(t)|^3}\cdot
\left(A_{rad}(Z(t)) + A_{tan}(Z(t))\right)\right|\leq
\ee
\[
\left| \frac{|\tilde Z(t)|^2 I_d- \tilde Z(t) \tilde Z^T(t)}{|\tilde Z(t)|^3}\cdot a(Z(t))|Z(t)|^{\beta-1}Z(t) \right|
+ K_1  \frac{|A_{tan}(Z(t))|}{|\tilde Z(t)|}\leq
\]
\[
\left| \frac{|\tilde Z(t)|^2 I_d- \tilde Z(t) \tilde Z^T(t)}{|\tilde Z(t)|^3}\cdot a(Z(t))|Z(t)|^{\beta-1}(\tilde Z(t) +\xi(t)) \right|
+ K_2  \frac{|Z(t)|^{\beta-\gamma}}{|\tilde Z(t)|}.
\]
Notice that for any $z\in \mbR^d:$
\[
(|z|^2 I_d -z z^T)z = |z|^2 z -z z^T z = |z|^2 z -z |z|^2 = 0.
\]
 So, the right hand side in \eqref{eq:1356} equals
\[
\left| \frac{|\tilde Z(t)|^2 I_d- \tilde Z(t) \tilde Z^T(t)}{|\tilde Z(t)|^3}\cdot a(Z(t))|Z(t)|^{\beta-1} \xi(t) \right|
+ K_2  \frac{|Z(t)|^{\beta-\gamma}}{|\tilde Z(t)|}\leq
\]
\[
K_3\left( \frac{1}{|\tilde Z(t)|}|Z(t)|^{\beta-1} |\xi(t)|
+  \frac{|Z(t)|^{\beta-\gamma}}{|\tilde Z(t)|} \right)\leq
\]
\[
K_4\left( \frac{ |\xi(t)|}{| Z(t)|^{2-\beta}}
+  \frac{1}{|Z(t)|^{1-\beta+\gamma}} \right)\leq
\]
\bel{eq:1357}
K_5\left( \frac{ 1+ t^{\frac{1}{1-\beta}-\delta}}{ t^{\frac{{2-\beta}}{1-\beta}}}
+  \frac{1}{ t^{\frac{{1-\beta+\gamma}}{1-\beta}}} \right)=
K_5\left( \frac{ 1}{ t^{1+\frac{{1}}{1-\beta}}}+
 \frac{ 1}{ t^{1+\delta}}
+  \frac{1}{ t^{1+\frac{{\gamma}}{1-\beta}}} \right).
\ee

This implies \eqref{eq:4.4.1}.

Lemma \ref{lem:4.2} is proved.
\end{proof}

Let us complete the proof of Theorem \ref{thm:integral_eq_growth}. It
follows from \eqref{eq:4.3.1} that it suffices  to verify that

\[\tilde r(t) \sim \left( (1-\beta) \bar a(\vf_{Z}(+\infty)) t  \right)^{\frac{1}{1-\beta}}  \mbox{ as } t\to\infty,
\]
where $\vf_{Z}(+\infty)$ is from Lemma \ref{lem:4.2}.

It follows from Lemmas \ref{lem:4.1}, \ref{lem:4.2}, equations \eqref{eq:4.1.2},
 \eqref{eq:4.3.1}, and \eqref{eq:298}  that
 \[
 \tilde r^{1-\beta}(t)=\tilde r^{1-\beta}(0)+
 \int_0^t
  [...] ds,
\]
where expression in brackets converges to $(1-\beta) \bar a(\vf_{Z}(+\infty))$ as $s\to\infty.$
Therefore
\[
\tilde r^{1-\beta}(t)\sim   (1-\beta) \bar a(\vf_{Z}(+\infty )) t,\ t\to\infty.
\]
Theorem \ref{thm:integral_eq_growth} is proved.
\end{proof}

\bigskip



\begin{proof}[Proof of Lemma \ref{lem:3.5}]
It is sufficient  to check the following condition on modulus of continuity, see \cite{Billingsley}:
\[
\forall T>0 \ \forall \mu>0\ \exists \delta>0\ \ \
\limsup_{\ve\to0}\Pb\left(
\exists s,t\in[0,T], \ |s-t|\leq\delta\ : \ \ \  | X_\ve(s)-X_\ve(t)|\geq \mu\right)<\mu.
\]


Since the vector field $A$ is of linear growth at the infinity, we have for any $T>0:$
\[
\lim_{M\to+\infty}\sup_{\ve\in (0,1]}\Pb(\sup_{t\in[0,T]}|X_\ve(t)|\geq M)=0.
\]
Denote
\[
K(\mu,M):=\sup_{\mu\leq |x|\leq M}|A(x)|.
\]

We have
\[
\Pb\big(\exists s,t\in[0,T], \ |s-t|\leq\delta\ : \ \ \   | X_\ve(s)-X_\ve(t)|\geq \mu\big)\leq
\]
\[
\Pb\left(\exists s,t\in[0,T], \ |s-t|\leq\delta\ : \ \ \   | \int_s^tA(X_\ve(z))dz| \geq \mu/4\ \mbox{ and } \
\mu/2\leq |X_\ve(z)|\leq M, z\in[s,t] \right)+
\]
\[
\Pb\big(\ve\sup_{s\in[0,T]}|B_\alpha(s)| \geq \mu/4\big)+ \sup_{\ve\in (0,1]}\Pb\big(\sup_{t\in[0,T]}|X_\ve(t)|\geq M\big)
\leq
\]
\[
\Pb\big(  \delta K(\mu/2, M) \geq \mu/4\big)+
\Pb\big(\sup_{s\in[0,T]}|B_\alpha(s)| \geq \mu/4\ve\big) + \sup_{\ve\in (0,1]}\Pb\big(\sup_{t\in[0,T]}|X_\ve(t)|\geq M\big)=
\]
\bel{eq:1441}
\Pb\big(\sup_{s\in[0,T]}|B_\alpha(s)| \geq \mu/4\ve\big) + \sup_{\ve\in (0,1]}\Pb\big(\sup_{t\in[0,T]}|X_\ve(t)|\geq M\big)
\ee
if $\delta < \frac{\mu}{4K(\mu/2, M)}.$

For any fixed $\mu $, the right-hand side of \eqref{eq:1441} can be made arbitrarily small if $M$ is sufficiently large and $\ve $ is sufficiently small. 
  This proves the Lemma.
\end{proof}

\begin{proof}[Proof of Lemma \ref{lem:3.4}.]
Introduce the following notations, see  \eqref{eq:space-time},
\[
\wt x:= x \ve^{\frac{-\alpha}{  \alpha+\beta-1}}, \ \ \  \ \wt  X(t):=\wt  X_{\ve, \wt  x}(t):= \ve^{\frac{-\alpha}{  \alpha+\beta-1}}   X_{\ve,x} (\ve^{\frac{\alpha(1-\beta)}{ \alpha+\beta-1}}t), \ \ 
B^{(\ve)}_\alpha(t):=\ve^{\frac{\beta-1}{  \alpha+\beta-1}} B_\alpha(\ve^{\frac{\alpha(1-\beta)}{ \alpha+\beta-1}}t)
\]

Then $\wt  X(t)$ satisfies \eqref{eq:657}.

It follows from \eqref{eq:3.6.-1} that  
\bel{eq:1498}
\frac{X_{\ve,x}(\tau_\delta^{\ve,x})}{|X_{\ve,x}(\tau_\delta^{\ve,x})|}= 
\frac{\wt  X_{\ve,\wt  x}(\wt  \tau_{\ve^{\frac{-\alpha}{\alpha+\beta-1}} \delta   }^{\ve,\wt  x})}{|\wt  X_{\ve,\wt  x}(\wt  \tau_{\ve^{\frac{-\alpha}{\alpha+\beta-1}} \delta   }^{\ve,\wt  x})|};
\ee
\bel{eq:1499}
\frac{|X_{\ve,x}(\tau_\delta^{\ve,x})|}{\delta} = 
\frac{|\wt  X_{\ve,\wt  x}(\wt  \tau_{\ve^{\frac{-\alpha}{\alpha+\beta-1}} \delta
   }^{\ve,\wt  x})|}{\ve^{\frac{-\alpha}{\alpha+\beta-1}} \delta}\ ,
 \ \ \ \ \tau_\delta^{\ve,x} = \ve^{\frac{\alpha(1-\beta)}{ \alpha+\beta-1}}\; \wt  \tau_{\ve^{\frac{-\alpha}{\alpha+\beta-1}} \delta   }^{\ve,\wt  x}
\ee

Further the arguments are similar to the  proof  of Theorem \ref{thm:integral_eq_growth}.
Set 
\[
\hat X_{\ve, \wt  x}(t):=\wt  X_{\ve, \wt  x}(t)- B^{(\ve)}_\alpha(t), \ \ 
\hat r_t:=|\hat X_{\ve, \wt  x}(t)|.
\]

Assume that $\omega$ is such that 
$|\wt  X_{\ve, \wt  x}(t)|\leq \hat r_t/2,
t\in [0, \wt  \tau_{\ve^{\frac{-\alpha}{\alpha+\beta-1}} \delta }^{\ve,\wt  x} ]
$, cf. \eqref{eq:assumption}.

If $R_0$ is sufficiently large, $\delta_2$ is sufficiently small and $R\geq R_0,\ \delta\in (0,\delta_2)$, then applying 
\eqref{eq:A_t_0}, \eqref{eq:A_rad_ve}, \eqref{eq:A_tan_ve}, and               \eqref{eq:4.1.2}  we get for this $\omega$:
\[
(\hat r(t))^{{1-\beta}}\geq K(\delta_2) t, t\in [0, \wt  \tau_{\ve^{\frac{-\alpha}{\alpha+\beta-1}} \delta }^{\ve,\wt  x} ],
\]
and moreover
\bel{eq:1529}
\wt  \tau_{\ve^{\frac{-\alpha}{\alpha+\beta-1}} \delta }^{\ve,\wt  x}\leq 
\frac{\ve^{\frac{-\alpha}{(1-\beta)(\alpha+\beta-1)}} \delta^{\frac{1}{1-\beta}}}{K}
\ee

We get \eqref{eq:lem3.5.2}, \eqref{eq:lem3.5.3} if we  use \eqref{eq:Khintchine}, 
 \eqref{eq:1499}  and reasoning of Lemma \ref{lem:4.1}.

To prove \eqref{eq:lem3.5.1} we have to apply \eqref{eq:A_t_0}, \eqref{eq:A_tan_ve}, \eqref{eq:1498} \eqref{eq:1529},   and  the arguments
 of Lemma \ref{lem:4.2} (see \eqref{eq:1356}, \eqref{eq:1357}).
\end{proof}

\bigskip

\begin{thm}\label{thm:Lip_expl}
Consider a system of ODEs in $\mbR\times\mbR^d:$
\bel{eq:RF}
\ba
dR(t)=a(R(t), \Phi(t)) R^\beta(t) dt, \\
d\Phi(t)=b(R(t), \Phi(t)) R^{\beta+\delta-1}(t) dt,
\ea
\ee
where $\delta\in(0,1),\ \beta\in (-\infty,1)$, and functions  $a: [0,\infty)\times\mbR^d\to\mbR$, $b: [0,\infty)\times\mbR^d\to\mbR^d$ are bounded and 
continuous.

Assume that

\noindent
1)  the function $a$ is positive and separated from zero:
\bel{eq:ass-drift}
\exists  a_\pm>0\ \forall r\geq 0, \ \vf\in \mbR^d: \ \ a_-\leq a(r,\vf)\leq a_+;
\ee
2) the functions  $a$ and $b$ are Lipschitz continuous in $\vf$ uniformly in $r: $
\[
\exists L\ \  \forall r\geq 0\ \vf_1, \vf_2\in \mbR^d: \ \ |a(r,\vf_1)- a(r,\vf_2)|+|b(r,\vf_1)- b(r,\vf_2)|\leq L |\vf_1-\vf_2|.
\]
Then for any initial starting point $(r, \vf)\in (0,\infty)\times\mbR^d$ there is a unique solution $(R_{r, \vf}(t), \Phi_{r, \vf}(t))$
to equation \eqref{eq:RF}. Moreover, the map
\[
 [0,\infty)\times (0,\infty)\times\mbR^d\ni (t, r, \vf)\to (R_{r, \vf}(t), \Phi_{r, \vf}(t))
\]
can be extended by continuity to $ [0,\infty)\times [0,\infty)\times\mbR^d$, and
 its extension $(R_{0, \vf}(t), \Phi_{0, \vf}(t))$ is a unique solution to  \eqref{eq:RF}
among all solutions satisfying  $R_{0, \vf}(t)>0$ for $t>0$.
\end{thm}
\begin{proof}[Proof of Theorem \ref{thm:Lip_expl}.]
Existence of solutions for $r>0, \vf\in\mbR^d$ follows from the Peano theorem. 

It follows from \eqref{eq:ass-drift} and the comparison with the  equation $d\rho_\pm(t) = a_\pm\rho_\pm^\beta(t) dt$
that
\bel{eq:bound_below}
 (a_- (1-\beta)t)^{\frac{1}{1-\beta}}\leq R_{r, \vf}(t)  \leq (r^{1-\beta} + a_+ (1-\beta)t)^{\frac{1}{1-\beta}}
\ee
  for  any $r>0, \vf\in\mbR^d, t\geq 0.$ Moreover, any solution  $R_{0, \vf}(t)$ 
 such that $R_{0, \vf}(t)>0, \ t>0$ also satisfies \eqref{eq:bound_below}.

It follows from the compactness arguments that there is a sequence $\{r_n\},\ \lim_{n\to\infty}  r_n=0$ 
such that the sequence  $\{(R_{r_n, \vf}(\cdot), \Phi_{r_n, \vf}(\cdot))\}$ converges uniformly on compact sets. It easy to see that
its limit $(R_{0, \vf}(t), \Phi_{0, \vf}(t))$ is a solution to \eqref{eq:RF} and satisfies  \eqref{eq:bound_below} with $r=0.$

Let us prove uniqueness. Apply transformation of time arguments, e.g. \cite[Chapter IV \S 4]{IW}.
Set
\[
\ba
 A_{r,\vf}(t):= \int_0^t a(R_{r,\vf}(z), \Phi_{r,\vf}(z)) R_{r,\vf}^\beta(z) dz,\\
 \tilde R_{r,\vf}(t):=R_{r,\vf}(A_{r,\vf}^{-1}(t)),\\
 \tilde \Phi_{r,\vf}(t):=\Phi_{r,\vf}(A_{r,\vf}^{-1}(t)),
\ea
\]
where $A_{r,\vf}^{-1}$ is the inverse function. The function  $A_{r,\vf}^{-1}$ is well defined  because $A_{r,\vf}$ is continuous and increasing function.

Then 
\[
d\tilde R_{r,\vf}(t)= dt,  
\]
\[
d\tilde \Phi_{r,\vf}(t)=\frac{b(\tilde R_{r,\vf}(t), \tilde \Phi_{r,\vf}(t)) \tilde R^{\beta+\delta-1}(t)}{a(\tilde R_{r,\vf}(t), \tilde \Phi_{r,\vf}(t)) \tilde  R^\beta(t)} dt 
=\frac{b(\tilde R_{r,\vf}(t), \tilde \Phi_{r,\vf}(t)) }{a(\tilde R_{r,\vf}(t), \tilde \Phi_{r,\vf}(t))} \tilde R_{r,\vf}^{\delta-1}(t) dt.
\]
Hence 
\[
\tilde R_{r,\vf}(t)= r+t,
\]
and
\bel{eq:Phi_tilde}
\tilde \Phi_{r,\vf}(t)= 
r+\int_0^t\frac{b(r+z, \tilde \Phi_{r,\vf}(z)) }{a(r+z, \tilde \Phi_{r,\vf}(z))}  (r+z)^{\delta-1} dz.
\ee
It follows from the assumptions of the Theorem that the function $(r,t,\vf)\to \frac{b(r+t, \vf) }{a(r+t, \vf)} $ is bounded and Lipschitz continuous  in $\vf$ uniformly in $r,t.$
Therefore, the local  integrability of $t\to(r+t)^{\delta-1}$ yields existence and uniqueness of a solution to \eqref{eq:Phi_tilde} for $r\geq0, \vf\in\mbR^d$, and 
continuity of $\tilde \Phi_{r,\vf}(t)$ in $(r,\vf,t)\in[0,\infty)\times[0,\infty)\times\mbR^d.$

To get $(R_{r,\vf}(t), \Phi_{r,\vf}(t))$ we have to make the inverse transformation of time. Set
\[
\tilde A_{r,\vf}(t):= \int_0^t a^{-1}(\tilde R_{r,\vf}(z), \tilde \Phi_{r,\vf}(z)) \tilde R_{r,\vf}^{-\beta}(z) dz.
\]
Notice that $\tilde A_{r,\vf}(t)$  is well defined and increasing in $t$. 

The functions
\[
\ba
R_{r,\vf}(t)=\tilde R_{r,\vf}(\tilde A_{r,\vf}^{-1}(t)),\\
\Phi_{r,\vf}(t)=\tilde \Phi_{r,\vf}(\tilde A_{r,\vf}^{-1}(t))
\ea
\]
satisfy \eqref{eq:RF}. Uniqueness of a solution to \eqref{eq:RF} follows from uniqueness for  \eqref{eq:Phi_tilde} and the fact that the
correspondence $(R_{r, \vf}, \Phi_{r, \vf}) \leftrightarrow (\tilde R_{r, \vf}, \tilde \Phi_{r, \vf})$
is one-to-one. 

The function $\tilde A_{r,\vf}(t)$ is continuous in $(r,\vf,t)\in[0,\infty)\times[0,\infty)\times\mbR^d$ due to the Lebesgue dominated convergence theorem and continuity of $\tilde \Phi_{r,\vf}(t)$. So, the inverse function $\tilde A^{-1}_{r,\vf}(t)$ is also continuous in $(r,\vf,t)$.
 This and continuity   of $(\tilde R_{r, \vf}, \tilde \Phi_{r, \vf})$ in $(r,\vf,t)$ yields the continuity of  $(R_{r, \vf}, \Phi_{r, \vf})$ .

\end{proof}

\bigskip

\thanks{{\bf Acknowledgements}.  A. Pilipenko 
was partially supported by the Alexander von Humboldt Foundation within 
the Research Group Linkage Programme {\it Singular diffusions: analytic and stochastic approaches}.
 The research of A. Pilipenko and F. Proske was carried out with support of the Senter for internasjonalisering av utdanning (SIU), within the project Norway-Ukrainian  Cooperation in Mathematical Education, project number CPEA-LT-2016/10139.}

\end{document}